\documentclass[12pt]{article}
\usepackage{amsmath,bbm,amsthm,amssymb,amstext,amscd,graphicx}
\def\mapright#1{\smash{\mathop{\longrightarrow}\limits^{#1}}}

\newtheorem{theorem}{Theorem}
\newtheorem{corollary}{Corollary}

\newtheorem{lemma}{Lemma}
\newtheorem{goals}{Goals}

\newtheorem{assumption}{Assumption}
\newtheorem{proposition}{Proposition}

\theoremstyle{definition}

\newtheorem{remark}{Remark}

\def\supp{\mathop{\rm supp}}
\def\diam{\mathop{\rm diam}}
\def\sp{\mathop{\rm sp}}

\def\codim{\mathop{\rm codim}}

\def\sech{{\rm sech}}

\def\Dx{\mathcal{D}_{X,\epsilon}}
\def\Dy{\mathcal{D}_{Y'_\epsilon,\epsilon}}
\def\Dxd{\mathcal{D}_{\dot{X},\epsilon}}
\def\Dyd{\mathcal{D}_{\dot{Y}'_\epsilon,\epsilon}}
\def\P{\mathcal{P}}

\def\Pe{\mathcal{P}_\epsilon}
\def\Le{\mathcal{L}_\epsilon}
\def\Led{\dot{\mathcal{L}}_\epsilon}
\def\ax{\alpha_{X,\epsilon}}
\def\ay{\alpha_{Y,\epsilon}}

\title{An analytic framework for identifying finite-time coherent sets in time-dependent dynamical systems}
\author{Gary Froyland\footnote{email: g.froyland@unsw.edu.au, tel: +61 2 9385 7050, fax: +61 2 9385 7123.} \\ School of Mathematics and Statistics \\ University of New South Wales \\ Sydney NSW 2052, Australia}
\begin{document}
\maketitle
\begin{abstract}
The study of transport and mixing processes in dynamical systems is particularly important for the analysis of mathematical models of physical systems.
Barriers to transport, which mitigate mixing, are currently the subject of intense study.
In the autonomous setting, the use of transfer operators (Perron-Frobenius operators) to identify invariant and almost-invariant sets has been particularly successful.
In the nonautonomous (time-dependent) setting, \emph{coherent sets}, a time-parameterised family of minimally dispersive sets, are a natural extension of almost-invariant sets.
The present work introduces a new analytic transfer operator construction that enables the calculation of \emph{finite-time} coherent sets (sets are that minimally dispersive over a finite time interval).
This new construction also elucidates the role of diffusion in the calculation and we show how properties such as the spectral gap and the regularity of singular vectors scale with noise amplitude.
The construction can also be applied to general Markov processes on continuous state space.
\end{abstract}
\section{Introduction}

Transport and mixing processes \cite{Meiss1992,aref_02,wiggins_05} play an important role in many
natural phenomena and their mathematical analysis has received
considerable interest in the last two decades.
\emph{Transport} is a key descriptor of the impact of the system's dynamics in the physical world, and \emph{coherent structures} and \emph{metastability} provide information on the multiple timescales of the system, which is crucial for efficient modelling approaches.
Persistent \emph{barriers to transport} play fundamental roles in geophysical systems by organising fluid flow and obstructing transport.
For example, ocean eddy boundaries strongly influence the horizontal distribution of heat in the ocean, and atmospheric vortices can trap chemicals and pollutants.
These time-dependent persistent transport barriers, or Lagrangian coherent structures,
are often difficult to detect and track by measurement (for example, by satellite observations), and many coherent structures present in the flow are very difficult to detect, map, and track to high precision by analytical means.

Coherent structures have been the subject of intense research for over a decade, primarily by {\em geometric} approaches.
The notion that geometric structures such as invariant manifolds
play a key role in dynamical transport and mixing for fluid-like
flow dates back at least two decades. In autonomous
settings, invariant cylinders and tori form impenetrable dynamical
barriers. This follows directly from the uniqueness of trajectories
of the underlying ordinary differential equation. Slow mixing
and transport in periodically driven maps and flows can sometimes
be explained by lobe dynamics of invariant manifolds \cite{romkedar_etal_1990,romkedar_wiggins_90}.
In non-periodic time-dependent settings, finite-time hyperbolic
material lines and surfaces, and more generally
\emph{Lagrangian coherent structures} (LCSs) \cite{Haller_00,Haller_01,shadden_lekien_marsden_05,Haller_11} have been put forward as a geometric approach to identifying barriers to mixing.
LCSs are material (ie.\ following the flow) co-dimension 1 objects that are locally the strongest repelling or attracting objects.
In practice these objects are often numerically estimated by a finite-time Lyapunov exponent (FTLE) field, with additional requirements \cite{Haller_11}. 
Recent work \cite{Haller_12} considers minimal stretching material lines.

{\em Probabilistic} approaches to studying transport, based around the {transfer operator} (or Perron-Frobenius operator) 
have been developed for autonomous systems, including \cite{dellnitz_junge_99,DFJ01,froyland_01,bollt_billings_schwartz_02,padberg_thesis_05}.
These approaches use the transfer operator (a linear operator providing a global description of the action of the flow on densities) to answer questions about global transport properties.
The transfer operator is particularly suited to identifying global metastable or almost-invariant structures in phase space \cite{dellnitz_junge_99,schuette_huisinga_deuflhard_01,FD03,froyland_05}.
Related methods, which attempt to decompose phase space into ergodic components include \cite{budivsic_2012,mezic_wiggins_99}.
Transfer operator-based methods have been very successful in resolving almost-invariant objects in a variety of applied settings, including molecular dynamics (to identify molecular conformations) \cite{schuette_huisinga_deuflhard_01}, ocean dynamics (to map oceanic gyres) \cite{froyland_padberg_england_treguier_07,DFHPS09}, steady or periodically forced fluid flow (to identify almost-cyclic and poorly mixing regions) \cite{singh_etal_09,stremler11}, and granular flow (to identify regions with differing dynamics) \cite{christov_etal_11}.

In the aperiodic time-dependence setting, the papers \cite{SFM10,FSM10} introduced the notion of finite-time coherent sets, which are natural time-dependent analogues of almost-invariant or metastable sets.
A pair of optimally coherent sets (one at an initial time, and one at a final time) have the property that the proportion of mass that is in the set at the initial time \emph{and} is mapped into the set at the final time is \emph{maximal}.
Thus, there is minimal leakage or dispersion from the sets over the finite interval of time considered.
The construction in \cite{FSM10} is a novel matrix-based approach for determining regions in phase space that are maximally coherent over a \emph{finite} period of time.
The main idea is to study the singular values and singular vectors of a discretised form of the transfer operator that describes the flow over the finite duration of interest.
Optimality properties of the singular vectors lead to the optimal coherence properties of the sets extracted from the singular vectors.
Moreover, the construction enabled the tracking of an \emph{arbitrary} probability measure, corresponding to an arbitrarily weighted flux.

Numerical experiments \cite{FSM10} have very successfully analysed an idealised stratospheric flow model, and
European Centre for Medium Range Weather Forecasting (ECMWF) data to delimit the stratospheric Antarctic polar vortex.
The method of \cite{FSM10} has also been successfully used to map and track Agulhas rings in three dimensions \cite{FHRSS12}, yielding eddy structures of greater coherence than standard Eulerian methods based on a snapshot of the underlying time-dependent flow.
These exciting numerical results  motivate further fundamental mathematical questions.
On the mathematical side, the numerical diffusion inherent in the numerical discretisation procedure \cite{FSM10}
 plays a role in the smoothness of the boundaries of the sets identified by the method.
On the practical side, dynamical systems models of geophysical processes are often a combination of deterministic (e.g. advective) and stochastic (e.g. diffusive) components.

The mathematical framework of \cite{FSM10} is purely finitary, which obscures the important role played by diffusion.
The present work extends the matrix-based numerical approach of \cite{FSM10} to an operator-theoretic framework in order to formalise the constructions in a transfer operator setting, and to enable an analytic understanding of the interplay between advection and diffusion.
This interplay has practical significance as mixing in geophysical models is often not resolved explicitly at subgrid scale by the model, but included via a ``parameterisation'' of subgrid diffusion.
Studies of ocean models \cite{england_rahmstorf_1999,knutti_et_al_2000} have shown that the thermohaline circulation and ventilation rates depend sensitively on subgrid parameter.
Similarly, transport barriers may sensitively depend on subgrid diffusion;  indeed the observed sensitivity to thermohaline circulation and ventilation rates may be due to sensitivity of hidden transport barriers.

%
%

In this work, we formalise the framework of \cite{FSM10} by casting the constructions as bounded linear operators;  specifically using Perron-Frobenius operators to effect the advective dynamics and diffusion operators to effect the diffusive dynamics.
This operator-theoretic setting clarifies the coherent set framework and isolates the important role played by diffusion in the computations.
As in \cite{FSM10}, there are no assumptions about stationarity, in fact, our machinery is specifically designed to handle nonautonomous, random, and non-stationary systems.

The finite-time constructions considered here are to be contrasted with recent time-asymptotic transfer operator constructions \cite{FLQ10,FLQ2,GTQ}, where equivariant Oseledets subspaces play the role of singular vectors in the present paper.
Time-asymptotic (or ``infinite-time'') coherent sets are those sets that disperse at the slowest (exponential rate) over an infinite time horizon;  numerical calculations may be found in \cite{FLS10}.
In the present paper we are specifically interested in sets that are most coherent over a particular finite-time horizon.


An outline of the paper is as follows.
Section 2 states the necessary optimality properties of singular vectors of compact linear operators on Hilbert space.
In Section 3 we specialise the linear operators to transfer operators defined by a stochastic kernel and describe the coherent set framework in this setting.
Section 4 considers the situation where the transfer operators are compositions of a Perron-Frobenius operator with a diffusion operator, and demonstrates that the diffusion operator is solely responsible for the subunit singular values.
Section 5 develops upper bounds for the regularity of the singular vectors as a function of the diffusion radius, and develops an upper bound for the spectral gap in terms of the diffusion radius.
Section 6 contains a case study of a nonautonomous idealised stratospheric flow.
We numerically demonstrate the effect of diffusion on the spectrum, the singular vectors, and resulting coherent sets, and then conclude.

\section{Singular Vectors for Compact Operators}
Let $\mathcal{L}:\mathcal{X}\to\mathcal{Y}$ be a compact linear mapping between Hilbert spaces.
We will shortly specialise to the setting where $\mathcal{L}$ is a transfer operator, but for the moment we only make use of compactness and the existence of an inner product.
Denote $\mathcal{A}=\mathcal{L}^*\mathcal{L}:\mathcal{X}\to\mathcal{X}$, where $\mathcal{L}^*$ is the dual of $\mathcal{L}$.
Note that $\mathcal{A}$ is compact, self-adjoint, and positive\footnote{Recall that an operator $\mathcal{A}$ on Hilbert space is called positive if $\langle \mathcal{A}\mathbf{x},\mathbf{x}\rangle\ge 0$ for all $\mathbf{x}$.}, and thus the spectrum of $\mathcal{A}$ is non-negative.
Enumerate the eigenvalues of $\mathcal{A}$, $\lambda_1\ge \lambda_2\ge\cdots\ge 0$
where the number of occurrences equals the multiplicity of the eigenvalue.
By the spectral theorem for compact self-adjoint operators (see eg.\ Theorem II.5.1, \cite{conway}), we can find an orthonormal basis of eigenvectors $u_l\in \mathcal{X}$, enumerated so that $\mathcal{A}u_l=\lambda_l u_l$, so that
\begin{equation}
\mathcal{A}=\sum_{l=1}^{M} \lambda_l\langle\cdot,u_l\rangle_\mathcal{X} u_l,
\end{equation}
where $M$ may be finite or infinite.

One has the following minimax principle (see eg.\ Theorem 9.2.4, p212 \cite{birman_solomjak}):
\begin{theorem}
\label{minimaxthm}
\begin{equation}
\label{minimax+eqn}
\lambda_l=\min_{V:\codim V\le l-1<M} \max_{0\neq \mathbf{x}\in V} \frac{\langle \mathcal{A}\mathbf{x},\mathbf{x}\rangle_\mathcal{X}}{\langle \mathbf{x},\mathbf{x}\rangle_\mathcal{X}}, l=1,\ldots,M.
\end{equation}
Furthermore, the maximising $\mathbf{x}$s are the $u_l$, $l=1,\ldots,M$.
\end{theorem}
From this, we obtain:
\begin{proposition}
\label{newminimaxprop}
\begin{equation}
\label{newminimax}
\sigma_l:=(\lambda_l)^{1/2}=\min_{V:\codim V\le l-1<M} \max_{0\neq \mathbf{x}\in V,0\neq \mathbf{y}\in \mathcal{Y}} \frac{\langle \mathcal{L}\mathbf{x},\mathbf{y}\rangle_\mathcal{Y}}{\|\mathbf{x}\|_\mathcal{X}\|\mathbf{y}\|_\mathcal{Y}}, l=1,\ldots,M.
\end{equation}
The maximising $\mathbf{x}$ is $u_l$ and the maximising $\mathbf{y}$ is $\mathcal{L}u_l/\|\mathcal{L}u_l\|_\mathcal{Y}$, $l=1,\ldots,M$.
\end{proposition}
\begin{proof}
Taking the square root of (\ref{minimax+eqn}) one has
\begin{equation*}
\frac{\langle \mathcal{A}\mathbf{x},\mathbf{x}\rangle_\mathcal{X}^{1/2}}{\langle \mathbf{x},\mathbf{x}\rangle_\mathcal{X}^{1/2}}=\frac{\langle \mathcal{L}\mathbf{x},\mathcal{L}\mathbf{x}\rangle_\mathcal{Y}^{1/2}}{\|\mathbf{x}\|_\mathcal{X}}=\frac{\|\mathcal{L}\mathbf{x}\|_\mathcal{Y}}{\|\mathbf{x}\|_\mathcal{X}}=\max_{0\neq \mathbf{y}\in\mathcal{Y}}\frac{\langle \mathcal{L}\mathbf{x},\mathbf{y}/\|\mathbf{y}\|_\mathcal{Y}\rangle_\mathcal{Y}}{\|\mathbf{x}\|_\mathcal{X}},
\end{equation*}
where the maximising $\mathbf{y}/\|\mathbf{y}\|_\mathcal{Y}$ is $\mathcal{L}\mathbf{x}/\|\mathcal{L}\mathbf{x}\|_\mathcal{Y}$.
\end{proof}
We will call the maximising unit $\mathbf{x}$ and $\mathbf{y}$ in (\ref{newminimax}) the \emph{left and right singular vectors} of $\mathcal{L}$, respectively.
The $\sigma_l$ are the \emph{singular values} of $\mathcal{L}$.

\section{Transfer Operators}
We now begin to be more specific about the objects in the previous section.
We introduce measure spaces $(X,\mathcal{B}_X,\mu)$ and $(Y,\mathcal{B}_Y,\nu)$, where we imagine $X$ as our domain at an initial time, with $\mu$ being a reference measure describing the mass distribution of the object of interest.
We now transform $\mu$ forward by some finite-time dynamics (involving advection and diffusion) to arrive at a probability measure $\nu$ describing the distribution at this later time. The set $Y$ is the support of this transformed measure $\nu$, and may be thought of as the ``reachable set'' for initial points in $X$.
In applications, $\mu$ may describe the distribution of air or water particles, for instance, and $\nu$ the distribution at a later time.

We set $\mathcal{X}=L^2(X,\mu)$ and $\mathcal{Y}=L^2(Y,\nu)$ and denote the standard inner product on $L^2(X,\mu)$ by $\langle\cdot,\cdot\rangle_\mu$ and the inner product on $L^2(Y,\nu)$ by $\langle\cdot,\cdot\rangle_\nu$.
We begin to place some assumptions on our transfer operator $\mathcal{L}:\mathcal{X}\to\mathcal{Y}$ and its dual.

\begin{assumption}
\label{Lasses}
\quad
\begin{enumerate}
\item $(\mathcal{L}f)(y)=\int k(x,y)f(x)\ d\mu(x)$, where $k\in L^2(X\times Y,\mathcal{B}_X\times\mathcal{B}_Y,\mu\times\nu)$ is non-negative,
\item $\mathcal{L}\mathbf{1}_X=\mathbf{1}_Y$, equivalently, $\int k(x,y)\ d\mu(x)=1$ for $\nu$-a.a $y$,
\item $\mathcal{L}^*\mathbf{1}_Y=\mathbf{1}_X$, equivalently, $\int k(x,y)\ d\nu(y)=1$ for $\mu$-a.a $x$,
\end{enumerate}
\end{assumption}


Non-negativity of the stochastic kernel $k$ in Assumption \ref{Lasses} (1) is a consistency requirement that says that if $f$ represents some distribution of mass with respect to $\mu$, then $\mathcal{L}f$ also represents some mass distribution.
Square integrability of $k$ will provide compactness of $\mathcal{L}$;  we discuss this shortly.
Assumption \ref{Lasses} (2) says that the function $\mathbf{1}$ (the density function for the measure $\mu$ in $L^2(X,\mu)$) is mapped to $\mathbf{1}$ (the density function for the measure $\nu$ in $L^2(Y,\nu)$). This is a normalisation condition on $\mathcal{L}$.
Assumption \ref{Lasses} (3) says that $\mathcal{L}$ preserves integrals.  That is, $\int_Y \mathcal{L}f\ d\nu=\int_X f\ d\mu$.  This follows since $\int_Y \mathcal{L}f\ d\nu=\langle \mathcal{L}f,\mathbf{1}\rangle_\nu=\langle f,\mathcal{L}^*\mathbf{1}\rangle_\mu=\int_X f\ d\mu$. 

\begin{lemma}
\label{compactlemma}
Under Assumption \ref{Lasses} (1), $\mathcal{L}:L^2(X,\mu)\to L^2(Y,\nu)$ and $\mathcal{L}^*:L^2(Y,\nu)\to L^2(X,\mu)$ are both compact operators.
\end{lemma}
\begin{proof}
See appendix.
\end{proof}

Compactness ensures that the spectrum of $\mathcal{L}^*\mathcal{L}=:\mathcal{A}:L^2(X,\mu)\circlearrowleft$ away from the origin is discrete;  in particular, the eigenvalue 1 is isolated and of finite multiplicity.
Further, self-adjointness and positivity of $\mathcal{A}$ implies the spectrum is non-negative and real, so the only eigenvalue of magnitude 1 is the eigenvalue 1 itself.

\begin{assumption}
\label{ass2}
The leading singular value $\sigma_1$ of $\mathcal{L}$ is simple.
\end{assumption}

We will show in Section 5 that the ``small random pertubation of a deterministic process'' kernel studied in Sections 4--6 satisfies Assumption \ref{ass2}.





\begin{proposition}
\label{Lsingvallemma}
Under Assumption \ref{Lasses},
\begin{enumerate}
\item the largest singular value of $\mathcal{L}$ is $\sigma_1=1$ and the corresponding left and right singular vectors are $\mathbf{1}_X$ and $\mathbf{1}_Y$, and
\item Under Assumption \ref{ass2},
\begin{equation}
\label{relaxed}
\sigma_2=\max_{f\in L^2(X,\mu),g\in L^2(Y,\nu)}\left\{\frac{\langle \mathcal{L}f,g\rangle_\nu}{\|f\|_\mu\|g\|_\nu}: \langle f,1\rangle_\mu=\langle g,1\rangle_\nu=0\right\}<1,
\end{equation} where the maximising $f$ and $g$ for are $u_2$ (the second largest singular vector for $\mathcal{L}$) and $\mathcal{L}u_2/\|\mathcal{L}u_2\|$, respectively.
\end{enumerate}
\end{proposition}
\begin{proof}
\quad
\begin{enumerate}
\item
By Lemma \ref{Lbound} (see Appendix) using Assumption \ref{Lasses} (2) and (3), one has $\|\mathcal{L}\|=\|\mathcal{L}^*\|\le 1$.
But $\mathcal{A}\mathbf{1}_X=\mathbf{1}_X$ so $\|\mathcal{A}\|=1$.
\item
By compactness $\sigma_1=1$ is isolated;  non-negativity of the spectrum of $\mathcal{A}$ and simplicity of $\sigma_1$ trivially implies $\sigma_2<1$.
To show that $\mbox{(\ref{relaxed})}=\sigma_2$, we note that by Assumption \ref{ass2}, when $l=2$ Proposition \ref{newminimaxprop} becomes
\begin{equation}
\label{applicProp1}\sigma^+_2=\min_{V\subset L^2(X,\mu):\codim V\le 1}\max_{f\in V, g\in L^2(Y,\nu)}\frac{\langle \mathcal{L}f,g\rangle_\nu}{\|f\|_\mu\|g\|_\nu}.
\end{equation}
The minimising codimension 1 subspace $V$ is $\sp\{\mathbf{1}\}^\perp$ (orthogonal to the leading singular vector $\mathbf{1}$).
Thus (\ref{applicProp1}) becomes
$$\max_{f\in L^2(X,\mu),g\in L^2(Y,\nu)}\left\{\frac{\langle \mathcal{L}f,g\rangle_\nu}{\|f\|_\mu\|g\|_\nu}:\langle f,\mathbf{1}\rangle_\mu=0 \right\}.$$
Arguing as in the proof of Proposition \ref{newminimaxprop} the maximising $g/\|g\|_\nu$ above is $\mathcal{L}f/\|\mathcal{L}f\|_\nu$ and as $\langle \mathcal{L}f,\mathbf{1} \rangle_\nu=\langle f,\mathbf{1}\rangle_\mu$ for any $f\in L^2(X,\mu)$ one may include the restriction $\langle g,\mathbf{1}\rangle_\nu=0$ without effect.
\end{enumerate}
%
\end{proof}




\subsection{Partitioning $X$ and $Y$}
We wish to measurably partition $X$ and $Y$ as $X_1\cup X_2$ and $Y_1\cup Y_2$ respectively, where
\begin{goals}
\label{goal}
\qquad
\begin{enumerate}
\item $\mathcal{L}\mathbf{1}_{X_k}\approx \mathbf{1}_{Y_k}$, $k=1,2$.
\item $\mu(X_k)=\nu(Y_k), k=1,2$,
\end{enumerate}
\end{goals}
The dynamical reasoning for these goals is that if $X_1$ is a coherent set, then one should be able to find a set $Y_1\subset Y$ that is approximately the image of $X_1$ under the (advective and diffusive) dynamics.
Similarly for $X_2$, the complement of $X_1$ in $X$.
While $Y_1$ is possibly only an approximate image of $X_1$, we will insist that there is no loss of mass under the dynamics (the action of $\mathcal{L}$).
Thus $\mu(X_1)=\nu(Y_1)$ and similarly for $X_2, Y_2$.
We use the shorthand $\{X_1,X_2\}\Bumpeq X$ to mean that $X_1$ and $X_2$ are a measurable partition of $X$.
%

Consider the \textbf{S}et-based problem:
\begin{eqnarray}
\nonumber\mbox{\textbf{(S)}}\ \max_{\{X_1,X_2\}\Bumpeq X,\{Y_1,Y_2\}\Bumpeq Y }&&\left\{\left\langle \mathcal{L}\left(\sqrt{\frac{\mu(X_2)}{\mu(X_1)}}\mathbf{1}_{X_1}-\sqrt{\frac{\mu(X_1)}{\mu(X_2)}}\mathbf{1}_{X_2}\right),\sqrt{\frac{\nu(Y_2)}{\nu(Y_1)}}\mathbf{1}_{Y_1}-\sqrt{\frac{\nu(Y_1)}{\nu(Y_2)}}\mathbf{1}_{Y_2}\right\rangle_\nu: \right. \\ \label{combinatorial}&&\qquad \left.\mu(X_k)=\nu(Y_k), k=1,2.\right\}.
\end{eqnarray}
We first briefly justify the expression \textbf{(S)} in terms of our Goals.
The constraints in (\ref{combinatorial}) directly capture $\mu(X_k)=\nu(Y_k), k=1,2$ from Goal \ref{goal}(2) above.
Noting that $\mu(X_k)=\nu(Y_k), k=1,2$, the objective may be rewritten as:
\begin{eqnarray}
\nonumber&=&\left(\frac{\mu(X_2)}{\mu(X_1)}\int\mathcal{L}\mathbf{1}_{X_1}\cdot\mathbf{1}_{Y_1}\ d\nu+\frac{\mu(X_1)}{\mu(X_2)}\int\mathcal{L}\mathbf{1}_{X_2}\cdot\mathbf{1}_{Y_2}\ d\nu\right)
-\left(\int\mathcal{L}\mathbf{1}_{X_1}\cdot\mathbf{1}_{Y_2}\ d\nu+\int\mathcal{L}\mathbf{1}_{X_2}\cdot\mathbf{1}_{Y_1}\ d\nu\right)\\
\nonumber&=&\left(\frac{\mu(X_2)}{\mu(X_1)}\langle\mathcal{L}\mathbf{1}_{X_1},\mathbf{1}_{Y_1}\rangle_\nu+\frac{\mu(X_1)}{\mu(X_2)}\langle\mathcal{L}\mathbf{1}_{X_2},\mathbf{1}_{Y_2}\rangle_\nu
\right)
-\left(\langle\mathcal{L}\mathbf{1}_{X_1},\mathbf{1}-\mathbf{1}_{Y_1}\rangle_\nu+\langle\mathcal{L}\mathbf{1}_{X_2},\mathbf{1}-\mathbf{1}_{Y_2}\rangle_\nu
\right)\\
\nonumber&=&\left(\frac{\mu(X_2)}{\mu(X_1)}+1\right)\langle\mathcal{L}\mathbf{1}_{X_1},\mathbf{1}_{Y_1}\rangle_\nu+\left(\frac{\mu(X_1)}{\mu(X_2)}+1\right)\langle\mathcal{L}\mathbf{1}_{X_2},\mathbf{1}_{Y_2}\rangle_\nu-\langle\mathbf{1}_{X_1},\mathcal{L}^*\mathbf{1}\rangle_\mu-\langle\mathbf{1}_{X_2},\mathcal{L}^*\mathbf{1}\rangle_\mu \\
\label{form1}&=&\frac{\langle\mathcal{L}\mathbf{1}_{X_1},\mathbf{1}_{Y_1}\rangle_\nu}{\mu(X_1)}+\frac{\langle\mathcal{L}\mathbf{1}_{X_2},\mathbf{1}_{Y_2}\rangle_\nu}{\mu(X_2)}-1.
\end{eqnarray}
Thus, the problem \textbf{(S)} is a natural way to achieve Goals \ref{goal} (1) and (2) above.

The set-based problem (\ref{combinatorial}) is difficult to solve because the functions $f$ and $g$ are restricted to particular forms (differences of two characteristic functions).
Using the shorthand $\psi_{X_1,X_2}=\sqrt{\frac{\mu(X_2)}{\mu(X_1)}}\mathbf{1}_{X_1}-\sqrt{\frac{\mu(X_1)}{\mu(X_2)}}\mathbf{1}_{X_2}$ and $\psi_{Y_1,Y_2}=
\sqrt{\frac{\nu(Y_2)}{\nu(Y_1)}}\mathbf{1}_{Y_1}-\sqrt{\frac{\nu(Y_1)}{\nu(Y_2)}}\mathbf{1}_{Y_2}$, one can easily check that $\|\psi_{X_1,X_2}\|_\mu=\|\psi_{Y_1,Y_2}\|_\nu=1$ and $\langle \psi_{X_1,X_2},\mathbf{1}\rangle_\mu=\langle \psi_{Y_1,Y_2},\mathbf{1}\rangle_\nu=0$.
Thus,
\begin{eqnarray*}
\mbox{\textbf{(S)}}&\le&\max_{\{X_1,X_2\}\Bumpeq X,\{Y_1,Y_2\}\Bumpeq Y }\langle \mathcal{L}\psi_{X_1,X_2},\psi_{Y_1,Y_2}\rangle_\nu \\
&\le& \max_{f\in L^2(X,\mu),g\in L^2(Y,\nu)}\left\{\frac{\langle \mathcal{L}f,g\rangle_\nu}{\|f\|_\mu\|g\|_\nu}: \langle f,1\rangle_\mu=\langle g,1\rangle_\nu=0\right\}=:\mbox{\textbf{(R)}},
\end{eqnarray*}
where we call \textbf{(R)} the \textbf{R}elaxed problem as it is a relaxation of the Set-based problem \textbf{(S)}.
Combining this with Proposition \ref{Lsingvallemma} we have
\begin{theorem}
\label{boundlemma}
Under Assumptions \ref{Lasses} and \ref{ass2},
\begin{equation}
\label{boundeqn}
\max_{\{X_1,X_2\}\Bumpeq X,\{Y_1,Y_2\}\Bumpeq Y }\left\{\frac{\langle\mathcal{L}\mathbf{1}_{X_1},\mathbf{1}_{Y_1}\rangle_\nu}{\mu(X_1)}+\frac{\langle\mathcal{L}\mathbf{1}_{X_2},\mathbf{1}_{Y_2}\rangle_\nu}{\mu(X_2)}:\mu(X_k)=\nu(Y_k), k=1,2\right\}\le 1+\sigma_2.
\end{equation}
\end{theorem}
This result is related to the upper bound in Theorem 2 \cite{huisinga_schmidt}.
The main difference is that \cite{huisinga_schmidt} is concerned with the dynamics of a reversible Markov operator at equilibrium and seeks fixed metastable sets, while here we have no assumption on reversibility of the dynamics, no restriction on the choice of $\mu$ ($\mu$ need not, for example, be invariant under the dynamics), and seek pairs of coherent sets $X_k,Y_k$, $k=1,2$, where the $X_k$ are not fixed, but map approximately onto $Y_k$ (which need not even intersect $X_k$).

Throughout this work, we have, for simplicity considered partitions of $X$ and $Y$ into two sets each.
Our ideas can be extended to multiple partitions using multiple singular vector pairs, as has been done in the autonomous setting with eigenvectors of transfer operators to find multiple almost-invariant sets \cite{FD03,gaveau_schulman_06,deuflhard_weber_05}.
As one progresses down the spectrum of sub-unit singular values, the corresponding singular vector pairs provide independent information on coherent separations of the phase space, which can in principle be combined using techniques drawn from the autonomous setting.


\begin{remark}
\label{remarkN}
In practice, solving \textbf{(R)} is straightforward once a suitable numerical approximation of $\mathcal{L}$ has been constructed.
Following \cite{FSM10} we will take the optimal $f$ and $g$ from (\ref{relaxed}) and use them heuristically to create partitions $\{X_1, X_2\}$ and $\{Y_1, Y_2\}$ via $X_1=\{f>b\}, X_2=\{f<b\}, Y_1=\{g>c\}, Y_2=\{g<c\}$, where $b$ and $c$ are chosen so that $\mu(X_k)=\nu(Y_k), k=1,2$.
This can be achieved, for example, by a line search on the value of $b$, where for each choice of $b$, there is a corresponding choice of $c$ that will match $\mu(X_k)=\nu(Y_k), k=1,2$;  we refer the reader to \cite{FLS10,FSM10} for implementation details.
\end{remark}

\section{Perron-Frobenius Operators on Smooth Manifolds}

We now consider the situation where the operator $\mathcal{L}$ arises from a Perron-Frobenius operator of a deterministic dynamical system.
Our dynamical system $T:M\to M$ acts on a compact subset\footnote{One could generalise the constructions in the sequel to $M$ a smooth Riemannian manifold in the obvious way.} $M\subset \mathbb{R}^d$.
The action of $T$ might be derived from either a discrete time or continuous time dynamical system.  In the latter case, $T$ will be a time-$t$ flow of some differential equation.  There are no assumptions about stationarity, in fact, our machinery is specifically designed to handle nonautonomous, random, and non-stationary systems.  We seek to analyse the one-step, finite-time action of $T$.

Our domain of interest will be $X\subset M$.
On $X$ we have a reference probability measure $\mu$ describing the quantity we are interested in tracking the transport of.
We assume that $\mu$ has an $L^2(X,\ell)$ density $h_\mu$ with respect to Lebesgue measure $\ell$.
We will discuss two situations:  purely deterministic dynamics, and deterministic dynamics preceded and followed by a small random perturbation.
In the former case we denote the set $Y$ by $Y_0=T(X)$ and the measure $\nu$ by $\nu_0$ (with density $h_{\nu_0}=d\nu_0/d\ell$), while in the latter case $Y$ is denoted $Y_\epsilon\supset T(X)$ and $\nu$ is denoted $\nu_\epsilon$ (with density $h_{\nu_\epsilon}=d\nu_\epsilon/d\ell$); this terminology emphasises whether or not a perturbation is present.
We define these objects formally in the coming subsections.
The same subscript notation will also shortly be extended to $\mathcal{L}$ and $\mathcal{A}$ with the obvious meanings.


Two notable aspects of our approach are (i) $\mu$ may be supported on a small subset $X$ of $M$, enabling one to neglect the remainder of the phase space $M$ for computations, and (ii) the possibility that $X\cap Y=\emptyset$;  thus $X$ need not be close to invariant, but in fact, all points may leave $X$ under the action of $T$.

\subsection{The deterministic setting}
Let $\ell$ denote Lebesgue measure on $M$ and suppose that $T:X\to Y_0$ is non-singular wrt Lebesgue measure (ie.\ $\ell(A)=0\Rightarrow \ell(T^{-1}A)=0$) for all measurable $A\subset Y_0$).
The evolution of a density $f\in L^1(X,\ell)$ under $T$ is described by the \emph{Perron-Frobenius operator} $\mathcal{P}:L^1(X,\ell)\to L^1(Y_0,\ell)$ defined by $\int_A \mathcal{P}f\ d\ell=\int_{T^{-1}A} f\ d\ell$ for all measurable $A\subset Y_0$.
If $T$ is differentiable, one may write $\mathcal{P}f(y)=\sum_{x\in T^{-1}y}\frac{f(x)}{|\det DT(x)|}$, however, in the remainder of this section we do not require differentiability of $T$.

If we define $\mathcal{L}_0:L^1(X,\mu)\to L^1(Y_0,\nu_0)$ by $\mathcal{L}_0f=\mathcal{P}(f\cdot h_\mu)/h_{\nu_0}$, where $h_{\nu_0}=\mathcal{P}h_\mu$, then $\mathcal{L}_0^*:L^\infty(Y_0,\nu_0)\to L^\infty(X,\mu)$ is given by $\mathcal{L}_0^*g=g\circ T$ since,
$$\langle \mathcal{L}_0f,g\rangle_{\nu_0}=\int \frac{\mathcal{P}(f\cdot h_\mu)}{h_{\nu_0}}\cdot g\ d{\nu_0}=\int \mathcal{P}(f\cdot h_\mu)\cdot g\ d\ell=\int f\cdot h_\mu\cdot g\circ T\ d\ell=\int f\cdot g\circ T\ d\mu=\langle f,\mathcal{L}_0^*g\rangle_\mu,$$
using standard $L^1/L^\infty$ duality of $\mathcal{P}$ (see eg. p48 \cite{lasota_mackey2}).

We now briefly argue that it is not instructive to use $\mathcal{L}=\mathcal{L}_0$ created directly from $\mathcal{P}$.
Firstly, if $T$ is invertible (eg.\ arising as a time-$t$ map of a flow), then a simple computation shows that $\mathcal{A}_0=\mathcal{L}_0^*\mathcal{L}_0$ is the identity operator and so there are no ``second largest'' eigenvalues of $\mathcal{A}_0$, only the eigenvalue 1.
Secondly, without the invertibility assumption, one can informally (due to non-compactness of $\mathcal{L}_0$) connect this dynamically with Theorem \ref{boundlemma}.
We note that the LHS of (\ref{boundeqn}) (substituting $\mathcal{L}=\mathcal{L}_0$) can be equivalently written as
\begin{eqnarray*}
\label{upperboundeqn2}
&&\max\left\{\frac{\langle\mathbf{1}_{X_1},\mathcal{L}_0^*\mathbf{1}_{Y_1}\rangle_\mu}{\mu(X_1)}+\frac{\langle\mathbf{1}_{X_2},\mathcal{L}_0^*\mathbf{1}_{Y_2}\rangle_\mu}{\mu(X_2)}: \{X_1,X_2\}\Bumpeq X, \{Y_1,Y_2\}\Bumpeq Y_0\right\}\\
&=&\max\left\{\frac{\langle\mathbf{1}_{X_1},\mathbf{1}_{T^{-1}Y_1}\rangle_\mu}{\mu(X_1)}+\frac{\langle\mathbf{1}_{X_2},\mathbf{1}_{T^{-1}Y_2}\rangle_\mu}{\mu(X_2)}: \{X_1,X_2\}\Bumpeq X, \{Y_1,Y_2\}\Bumpeq Y_0\right\}\\
&=&\max\left\{\frac{\mu(X_1\cap T^{-1}Y_1)}{\mu(X_1)}+\frac{\mu(X_2\cap T^{-1}Y_2)}{\mu(X_2)}: \{X_1,X_2\}\Bumpeq X, \{Y_1,Y_2\}\Bumpeq Y_0\right\}
\end{eqnarray*}
By choosing $\{Y_1,Y_2\}$ to be any nontrivial measurable partition and $X_k=T^{-1}(Y_k), k=1,2$, the value of the above expression becomes 2, forcing $\sigma_2=1$.
Thus, we see that defining $\mathcal{L}=\mathcal{L}_0$ using the deterministic Perron-Frobenius operator $\mathcal{P}$ does not allow us to find a ``distinguished'' coherent partition;  all partitions of this form are equally coherent.
In practice, one is typically interested in coherent sets that are robust in the presence of noise of a certain amplitude, or in the presence of diffusion inherent in a dynamical or physical model.

\subsection{Small random perturbations}
\label{diffsection}
To create ``distinguished'' coherent sets in purely advective dynamics, indicated by isolated singular values close to 1, we construct $\mathcal{L}=\mathcal{L}_\epsilon$ from the Perron-Frobenius operator and diffusion operators.
We will apply diffusion before and after the application of the Perron-Frobenius operator.
For $X\subset X_\epsilon\subset M$, we define a local diffusion operator $\Dx:L^1(X,\ell)\to L^1(X_\epsilon,\ell)$ via a bounded stochastic kernel: $\Dx g(y)=\int_X \ax(y-x)g(x)\ dx$, where $\ax:M\to \mathbb{R}^+$ satisfies $\int_{X_\epsilon} \ax(y-x)\ dy=1$ for $x\in X$.
Similarly, for $Y'_\epsilon\subset Y_\epsilon\subset M$ we define a local diffusion operator $\Dy:L^1(Y'_\epsilon,\ell)\to L^1(Y_\epsilon,\ell)$ by $\Dy g(y)=\int_{Y'_\epsilon} \ay(y-x)g(x)\ dx$, where $\ay$ is bounded and satisfies $\int_{Y_\epsilon} \ay(y-x)\ dy=1$ for $x\in Y'_\epsilon$.
We have in mind that $\ax$ and $\ay$ are supported in an $\epsilon$-neighbourhood of the origin, and that $X_\epsilon=\supp(\Dx\mathbf{1}_X)$, $Y'_\epsilon=T(X_\epsilon)$, and $Y_\epsilon=\supp(\Dx \mathbf{1}_{Y'_\epsilon})$.
Thus, we have
\begin{equation}
\label{floweqn}
L^1(X,\ell)\mapright{\Dx}L^1(X_\epsilon,\ell)\mapright{\mathcal{P}}L^1(Y'_\epsilon,\ell)\mapright{\Dy}L^1(Y_\epsilon,\ell)
\end{equation}


%
%

We now define $\mathcal{P}_\epsilon:L^1(X,\ell)\to L^1(Y_\epsilon,\ell)$ by $\mathcal{P}_\epsilon f(y):=\Dy\mathcal{P}\Dx f(y)$ and note that
\begin{eqnarray*}
\Dx\mathcal{P}\Dy f(y)&=&\int_{Y'_\epsilon} \ay(y-x)(\mathcal{P}\Dx f)(x)\ dx \\
&=&\int_{Y'_\epsilon} \ay(y-x)\mathcal{P}\left(\int_X \ax(x-z)f(z)\ dz\right)\ dx \\
&=&\int_{X_\epsilon} \ay(y-Tx)\left(\int_X \ax(x-z)f(z)\ dz\right)\ dx \qquad\mbox{by duality of $\mathcal{P}$} \\
&=&\int_X \left(\int_{X_\epsilon} \ay(y-Tx)\ax(x-z)\ dx\right)f(z)\ dz
\end{eqnarray*}
We note that to define $\mathcal{P}_\epsilon$ via the final displayed expression above, $T$ need only be measurable;  in particular, $T$ need not be nonsingular (nor invertible).

Finally, we define an operator $\mathcal{L}_\epsilon$ (which, under suitable conditions on $\ax$ and $\ay$ discussed shortly, will act on functions in $L^2(X,\mu)$) by
\begin{equation}
\label{specLeqn}\mathcal{L}_\epsilon f(y)=\frac{\mathcal{P}_\epsilon(f\cdot h_\mu)}{\mathcal{P}_\epsilon h_\mu}=\frac{\int_X \left(\int_{X_\epsilon} \ay(y-Tz)\ax(z-x)\ dz\right)f(x)h_\mu(x)\ dx}{\int_X \left(\int_{X_\epsilon} \ay(y-Tz)\ax(z-x)\ dz\right)h_\mu(x)\ dx}=\int_X k_\epsilon(x,y)f(x)\ d\mu(x),
\end{equation}
where $k_\epsilon(x,y)=\left(\int_{X_\epsilon} \ay(y-Tz)\ax(z-x)\ dz\right)/\int_X\left(\int_{X_\epsilon} \ay(y-Tz)\ax(z-x)\ dz\right)\ d\mu(x)$.
We denote the normalising term in the denominator $\int_X\left(\int_{X_\epsilon} \ay(y-Tz)\ax(z-x)\ dz\right)\ d\mu(x)=\mathcal{P}_\epsilon h_\mu(y)$ as $h_{\nu_\epsilon}(y)$, and define $\nu_\epsilon=dh_{\nu_\epsilon}/d\ell$.
By Lemma \ref{compactlemma}, if $k_\epsilon(x,y)\in L^2(X\times Y_\epsilon,\mu\times\nu_\epsilon)$ then $\mathcal{L}_\epsilon:L^2(X,\mu)\to L^2(Y_\epsilon,\nu_\epsilon)$ and is compact.

We note that the dual operators $\Dx^*:L^\infty(X_\epsilon,\ell)\to L^\infty(X,\ell)$ and $\Dy^*:L^\infty(Y_\epsilon,\ell)\to L^\infty(Y'_\epsilon,\ell)$ are given by $\Dx^* f(x)=\int_{X_\epsilon} \ax(x-y)f(y)\ dy$ and $\Dy^* g(x)=\int_{Y_\epsilon} \ay(x-y)g(y)\ dy$, respectively (in fact, since $\ax$ and $\ay$ are bounded, these dual operators may be considered to act on $L^1$ functions).
It is straightforward to show that the dual operator $\mathcal{L}_\epsilon^*:L^2(Y_\epsilon,\nu_\epsilon)\to L^2(X,\mu)$ is given by $\mathcal{L}_\epsilon^*g(x)=\int_{Y_\epsilon} k_\epsilon(x,y)g(y)\ d\nu_\epsilon(y)$, and substituting $\ax, \ay$ as above we have
\begin{eqnarray}
\mathcal{L}^*_\epsilon g(x)&=&\int_{Y_\epsilon} \frac{\left(\int_{X_\epsilon} \ay(y-Tz)\ax(z-x)\ dz\right)g(y)\ d\nu_\epsilon(y)}{h_{\nu_\epsilon}(y)}\\
&=&\int_{Y_\epsilon} \left(\int_{X_\epsilon} \ay(y-Tz)\ax(z-x)\ dz\right)g(y)\ dy\\
\label{specdualLeqn}&=&\int_{X_\epsilon} \ax(z-x)\left(\int_{Y_\epsilon} \ay(y-Tz)\ g(y)\ dy\right)\ dz\\
&=&\int_{X_\epsilon} \ax(z-x)\mathcal{K}(\Dy^* g)(z)\ dz\\
&=&\Dx^*\circ\mathcal{K}\circ\Dy^* g(x)
\end{eqnarray}
where $\mathcal{K}g=g\circ T$ is the Koopman operator.

We now verify that $\mathcal{L}_\epsilon$ satisfies Assumption \ref{Lasses}.
For Assumption \ref{Lasses}(2), we set $f=\mathbf{1}_X$ in (\ref{specLeqn}) and
for Assumption \ref{Lasses}(3), we set $g=\mathbf{1}_{Y_\epsilon}$ in (\ref{specdualLeqn}).
Whether $\mathcal{L}_\epsilon$ is compact or not (Assumption \ref{Lasses}(1)) depends on the form of $\ax$ and $\ay$;  we show that $\mathcal{L}_\epsilon$ is compact for a natural choice of $\ax, \ay$, namely $\ax=\ay=\mathbf{1}_{B_\epsilon(0)}/\ell(B_\epsilon(0))$\footnote{In the case where $M$ has boundaries so that $B_\epsilon(Tx)\nsubseteq M$ for some $x\in X$, we assume that the kernel $\alpha_\epsilon$ is suitably modified at these boundaries to remain stochastic.}.
This choice of $\ax, \ay$ means that the action of $\mathcal{D}_\epsilon$ is an averaging over a local $\epsilon$-ball, and the action of $\mathcal{P}_\epsilon$ may be interpreted as a small random perturbation with uniform density on an $\epsilon$-ball, followed by iteration by $T$, followed by another small random perturbation with uniform density on an $\epsilon$-ball.
Explicitly, for this choice of $\alpha_\epsilon$ one has
\begin{eqnarray}
\nonumber\mathcal{L}_\epsilon f(y)&=&\frac{\int_X\left(\int_{X_\epsilon}\mathbf{1}_{B_\epsilon(0)}(y-Tz) \mathbf{1}_{B_\epsilon(0)}(z-x)\ dz\right)f(x)\ d\mu(x)}{\int_X\left(\int_{X_\epsilon}\mathbf{1}_{B_\epsilon(0)}(y-Tz) \mathbf{1}_{B_\epsilon(0)}(z-x)\ dz\right)\ d\mu(x)}\\
\nonumber&=&\frac{\int_X\left(\int_{X_\epsilon}\mathbf{1}_{B_\epsilon(y)}(Tz) \mathbf{1}_{B_\epsilon(x)}(z)\ dz\right)f(x)\ d\mu(x)}{\int_X\left(\int_{X_\epsilon}\mathbf{1}_{B_\epsilon(y)}(Tz) \mathbf{1}_{B_\epsilon(x)}(z)\ dz\right)\ d\mu(x)}\\
\nonumber&=&\frac{\int_X\left(\int_{X_\epsilon}\mathbf{1}_{T^{-1}B_\epsilon(y)}(z) \mathbf{1}_{B_\epsilon(x)}(z)\ dz\right)f(x)\ d\mu(x)}{\int_X\left(\int_{X_\epsilon}\mathbf{1}_{T^{-1}B_\epsilon(y)}(z) \mathbf{1}_{B_\epsilon(x)}(z)\ dz\right)\ d\mu(x)}\\
\label{withinvball}&=&\frac{\int_X\ell(B_\epsilon(x)\cap T^{-1}B_\epsilon(y))f(x)\ d\mu(x)}{\int_X\ell(B_\epsilon(x)\cap T^{-1}B_\epsilon(y))\ d\mu(x)},\qquad\mbox{since $B_\epsilon(x)\subset X_\epsilon$.}
\end{eqnarray}

Thus, $\mathcal{L}_\epsilon f$ at $y\in Y_\epsilon$ is a double-average over the pull-back by $T$ of an $\epsilon$-neighbourhood of $y$.
Clearly, $\mathcal{L}_\epsilon \mathbf{1}_X=\mathbf{1}_{Y_\epsilon}$.
Also, by (\ref{specdualLeqn}):
\begin{eqnarray}
\nonumber\mathcal{L}_\epsilon^*g(x)&=&\frac{1}{\ell(B_\epsilon(0))^2}\int_{Y_\epsilon}\int_{X_\epsilon} \mathbf{1}_{B_\epsilon(0)}(y-Tz)\mathbf{1}_{B_\epsilon(0)}(z-x)\ dz\  g(y)\ dy\\
\nonumber&=&\frac{1}{\ell(B_\epsilon(0))}\int_{X_\epsilon}\mathbf{1}_{B_\epsilon(x)}(z)\left(\frac{1}{\ell(B_\epsilon(0))}\int_{Y_\epsilon} \mathbf{1}_{B_\epsilon(Tz)}(y)g(y)\  dy\right)\ dz \\
\label{withinvballdual}&=&\frac{1}{\ell(B_\epsilon(0))}\int_{X_\epsilon}\mathbf{1}_{B_\epsilon(x)}(z)\left(\frac{1}{\ell(B_\epsilon(0))}\int_{B_\epsilon(Tz)}g(y)\  dy\right)\ dz
\end{eqnarray}
noting that $B_\epsilon(Tz)\subset Y_\epsilon$ for $z\in X_\epsilon$ and $B_\epsilon(x)\subset X_\epsilon$ for $x\in X$.
Using these facts, one may compute that $\mathcal{L}^*_\epsilon\mathbf{1}_{Y_\epsilon}(x)=\mathbf{1}_X$.
\begin{lemma}
\label{compactlemma2}
If $\ell({Y_\epsilon})<\infty$ and $\ax(x)=\ay(x)=\mathbf{1}_{B_\epsilon(0)}(x)/\ell(B_\epsilon(0))$ then $k_\epsilon\in L^2(X\times {Y_\epsilon},\mu\times{\nu_\epsilon})$ and thus $\mathcal{L}_\epsilon$ is compact.
\end{lemma}
\begin{proof}
We show that $k_\epsilon(x,y)\in L^2(X\times {Y_\epsilon},\mu\times{\nu_\epsilon})$;  the result will then follow from Lemma \ref{compactlemma}.
We use the fact that $\int_{X_\epsilon} \ay(y-Tz)\ax(z-x)\ dz=\ell(B_\epsilon(x)\cap T^{-1}B_\epsilon(y))/\ell(B_\epsilon)^2$ as shown in (\ref{withinvball}).
\begin{eqnarray*}
\|k\|^2&=&\int_{Y_\epsilon}\int_X \frac{\left(\int_{X_\epsilon} \ay(y-Tz)\ax(z-x)\ dz\right)^2}{\left(\int_X\int_{X_\epsilon} \ay(y-Tz)\ax(z-x)\ dz\ d\mu(x)\right)^2}\ d\mu(x)d{\nu_\epsilon}(y)\\
&=&\int_{Y_\epsilon}\int_X \frac{\left(\int_{X_\epsilon} \ay(y-Tz)\ax(z-x)\ dz\right)^2}{\left(\int_X\int_{X_\epsilon} \ay(y-Tz)\ax(z-x)\ dz\ d\mu(x)\right)}\ d\mu(x)dy\\
&=&\frac{1}{\ell(B_\epsilon(0))^2}\int_{Y_\epsilon}\int_X \frac{\ell(T^{-1}B_\epsilon(y)\cap B_\epsilon(x))^2}{\int_X \ell(T^{-1}B_\epsilon(y)\cap B_\epsilon(x))\ d\mu(x)}\ d\mu(x)dy\\
&=&\frac{1}{\ell(B_\epsilon(0))^2}\int_{Y_\epsilon} \frac{\int_X\ell(T^{-1}B_\epsilon(y)\cap B_\epsilon(x))^2\ d\mu(x)}{\int_X \ell(T^{-1}B_\epsilon(y)\cap B_\epsilon(x))\ d\mu(x)}\ dy\\
&\le&\frac{1}{\ell(B_\epsilon(0))}\int_{Y_\epsilon} \frac{\int_X\ell(T^{-1}B_\epsilon(y)\cap B_\epsilon(x))\ d\mu(x)}{\int_X \ell(T^{-1}B_\epsilon(y)\cap B_\epsilon(x))\ d\mu(x)}\ dy\\
&\le&\frac{\ell(Y_\epsilon)}{\ell(B_\epsilon(0))}<\infty
%
\end{eqnarray*}
By Lemma \ref{compactlemma} we have that $\mathcal{L}_\epsilon$ and $\mathcal{L}_\epsilon^*$ are both compact.
\end{proof}

\begin{corollary}
If $\ell(Y_\epsilon)<\infty$ and $\ax(x)=\ay(x)=\mathbf{1}_{B_\epsilon(0)}(x)/\ell(B_\epsilon(0))$, the operator $\Le$ satisfies Assumption \ref{Lasses}.
\end{corollary}
In Section 5 we show that the above choice of $\ax,\ay$ satisfies Assumption \ref{ass2}.

\subsection{Objectivity}

We demonstrate that our analytic framework for identifying finite-time coherent sets is \emph{objective} or \emph{frame-invariant}, meaning that the method produces the same features when subjected to time-dependent ``proper orthogonal + translational'' transformations; see \cite{truesdellnoll,hallerobjective}.

In continuous time, to test for objectivity, one makes a time-dependent coordinate change $x\mapsto Q(t)x+b(t)$ where $Q(t)$ is a proper othogonal linear transformation and $b(t)$ is a translation vector, for $t\in[t_0,t_1]$.
The discrete time analogue is to consider our initial domain $X$ transformed to $\dot{X}=\{Q(t_0)x+b(t_0):x\in X\}$ and our final domain $Y_\epsilon$ transformed to $\dot{Y}_\epsilon=\{Q(t_1)y+b(t_1):y\in Y_\epsilon\}$.
For shorthand, we use the notation $\Phi_{t_0}$ and $\Phi_{t_1}$ for these transformations, so that $\dot{X}_\epsilon=\Phi_{t_0}(X)$ and $\dot{Y}_\epsilon=\Phi_{t_1}(Y_\epsilon)$.
The deterministic transformation $\dot{T}:\Phi_{t_0}(M)\to \Phi_{t_1}(M)$, which we wish to analyse using our transfer operator framework, is given by $\dot{T}=\Phi_{t_1}\circ T\circ \Phi_{t_0}^{-1}$.
The change of frames is summarised in the commutative diagram below.
$$
\begin{CD}
M @>T>> M\\
@VV\Phi_{t_0}V @VV\Phi_{t_1}V\\
\Phi_{t_0}(M) @>\dot{T}>> \Phi_{t_1}(M)
\end{CD}
$$
We define $\dot{\mu}=\mu\circ \Phi_{t_0}^{-1}$ and $\dot{\nu}_\epsilon=\nu_\epsilon\circ \Phi_{t_1}^{-1}$ as the transformed versions of $\mu$ and $\nu_\epsilon$;  $\dot{\mu}$ and $\dot{\nu}_\epsilon$ are probability measures on $\dot{X}$ and $\dot{Y}$, respectively.
We further define the Perron-Frobenius operators for $\Phi_{t_0}$ and $\Phi_{t_1}$, namely $\P_{\Phi_{t_0}}:L^2(X,\mu)\to L^2(\dot{X},\dot{\mu})$ and $\P_{\Phi_{t_1}}:L^2(Y,\nu_\epsilon)\to L^2(\dot{Y},\dot{\nu}_\epsilon)$ by $\P_{\Phi_{t_0}}f=f\circ \Phi_{t_0}^{-1}$ and $\P_{\Phi_{t_1}}f=f\circ \Phi_{t_1}^{-1}$, respectively.

The operators $\Dxd$ and $\Dyd$ are defined on $\dot{X}$ and $\dot{Y}'_\epsilon$ analogously to the definitions of $\Dx$ and $\Dy$; that is, $\Dxd g(y)=\int_{\dot{X}}\ax(y-x)g(x)\ dx$ and $\Dyd g(y)=\int_{\dot{Y'}_\epsilon} \ay(y-x)g(x)\ dx$, where $\ax=\ax=\mathbf{1}_{B_\epsilon(0)}/\ell(B_\epsilon(0))$.
Note in particular, that these operators are created \emph{independently} of the operators $\Dx$ and $\Dy$, and that $\Dxd$ and $\Dyd$ are \emph{not} simply defined by direct transformation of $\Dx$ and $\Dy$ with $\P_{\Phi_{t_0}},\P_{\Phi_{t_1}}$.
While the transformed deterministic dynamics $\dot{T}$ must of course be defined by conjugation with $\Phi_{t_0},\Phi_{t_1}$, our intention with the ``small random perturbation'' model is to use the \emph{same} diffusion operators $\Dx$, $\Dy$, \emph{without} any knowledge of coordinate changes, to construct an $\Led$.
Note that because $\Phi_{t_0}$ and $\Phi_{t_1}$ are proper othogonal affine transformations, one has
\begin{equation}
\label{commutativediff}
\P_{\Phi_{t_0}}\circ\Dx=\Dxd\circ\P_{\Phi_{t_0}}\mbox{ and } \P_{\Phi_{t_1}}\circ\Dy=\Dyd\circ\P_{\Phi_{t_1}};
  \end{equation}
geometrically, the LHSs are an averaging over $\epsilon$-balls followed by the transformations $\Phi_{t_0}$ (resp.\ $\Phi_{t_1}$), while the RHSs apply the transformation $\Phi_{t_0}$ (resp.\ $\Phi_{t_1}$) first and then average.
The result is the same because $\Phi_{t_0}^{-1}(B_\epsilon(x))=B_\epsilon(\Phi_{t_0}^{-1}(x))$;  similarly, for $\Phi_{t_1}$.

The Perron-Frobenius operator for $\dot{T}$, denoted $\P_{\dot{T}}$ is $\P_{\dot{T}}=\P_{\Phi_{t_1}}\circ\P\circ\P_{\Phi_{t_0}}^{-1}$.
Denote $h_{\dot{\mu}}=\P_{\Phi_{t_0}}h_\mu$.
Then, $\Led:L^2(\dot{X},\dot{\mu})\to L^2(\dot{Y}_\epsilon,\dot{\nu})$ is defined by $\Led f=\Dyd\circ\P_{\dot{T}}\circ \Dxd(f\cdot h_{\dot{\mu}})/\Dyd\circ\P_{\dot{T}}\circ \Dxd(h_{\dot{\mu}})$.

\begin{theorem}
The operator $\Le$ in the original frame and the operator $\Led$ in the transformed frame satisfy the commutative diagram:
$$
\begin{CD}
L^2(X,\mu) @>\Le>> L^2(Y_\epsilon,\nu_\epsilon)\\
@VV\P_{\Phi_{t_0}}V @VV\P_{\Phi_{t_1}}V\\
L^2(\dot{X},\dot{\mu})@>\Led>>L^2(\dot{Y}_\epsilon,\dot{\nu}_\epsilon)
\end{CD}.
$$
\end{theorem}
\begin{proof}
Note that because $\P_{\Phi_{t_0}}$ is a composition operator, one has $\P_{\Phi_{t_0}}(f\cdot g)=\P_{\Phi_{t_0}}f \cdot \P_{\Phi_{t_0}}g$ (and similarly for $\P_{\Phi_{t_1}}$);  we use this fact below.
We also make use of (\ref{commutativediff}).
\begin{eqnarray*}
\Led f&=&\Dyd\circ\P_{\dot{T}}\circ \Dxd(f\cdot h_{\dot{\mu}})/\Dyd\circ\P_{\dot{T}}\circ \Dxd(h_{\dot{\mu}})\\
&=&\Dyd\circ \P_{\Phi_{t_1}}\circ\P\circ\P_{\Phi_{t_0}}^{-1}\circ \Dxd(f\cdot h_{\dot{\mu}})/\Dyd\circ\P_{\Phi_{t_1}}\circ\P\circ\P_{\Phi_{t_0}}^{-1}\circ \Dxd(h_{\dot{\mu}})\\
&=&\P_{\Phi_{t_1}}\circ\Dy\circ\P\circ\Dx\circ\P_{\Phi_{t_0}}^{-1}(f\cdot h_{\dot{\mu}})/\P_{\Phi_{t_1}}\circ\Dy\circ\P\circ \Dx\circ\P_{\Phi_{t_0}}^{-1}(h_{\dot{\mu}})\\
&=&\P_{\Phi_{t_1}}\circ\Pe\circ\P_{\Phi_{t_0}}^{-1}(f\cdot h_{\dot{\mu}})/\P_{\Phi_{t_1}}\circ\Pe\circ\P_{\Phi_{t_0}}^{-1}(h_{\dot{\mu}})\\
&=&\P_{\Phi_{t_1}}\circ\Pe(\P_{\Phi_{t_0}}^{-1}f\cdot \P_{\Phi_{t_0}}^{-1}h_{\dot{\mu}})/\P_{\Phi_{t_1}}\circ\Pe(\P_{\Phi_{t_0}}^{-1}(h_{\dot{\mu}}))\\
&=&\P_{\Phi_{t_1}}\circ\Pe(\P_{\Phi_{t_0}}^{-1}f\cdot h_{{\mu}})/\P_{\Phi_{t_1}}\circ\Pe(h_{{\mu}})\\
&=&\P_{\Phi_{t_1}}\left(\Pe(\P_{\Phi_{t_0}}^{-1}f\cdot h_{{\mu}})/\Pe(h_{{\mu}})\right)\\
&=&\P_{\Phi_{t_1}}\left(\Le(\P_{\Phi_{t_0}}^{-1}f)\right),
\end{eqnarray*}
as required.
\end{proof}

\begin{corollary}
\label{dotcor}
If $\Le f=\lambda g$ where $f$ and $g$ are left and right singular vectors of $\Le$, respectively, then $\Led \P_{\Phi_{t_0}}f=\lambda \P_{\Phi_{t_1}}g$ where $\P_{\Phi_{t_0}}f$ and $\P_{\Phi_{t_1}}g$ are left and right singular vectors of $\Led$, respectively.
\end{corollary}
It follows from Corollary \ref{dotcor} that the coherent sets extracted on $\dot{X}$ and $\dot{Y}_\epsilon$ as eg.\ level sets from the singular vectors of $\Led$ will be transformed versions (under $\Phi_{t_0}$ and $\Phi_{t_1}$) of those extracted from $\Le$ on $X$ and $Y_\epsilon$, as required for objectivity.

More generally, if $\Dx$ and $\Dy$ are compact operators representing small diffusion, then $\P_{\Phi_{t_0}}\circ\Dx=\Dxd\circ\P_{\Phi_{t_0}}$ and $\P_{\Phi_{t_1}}\circ\Dy=\Dyd\circ\P_{\Phi_{t_1}}$ is a sufficient condition for the method to be objective.

\section{The case of $T$ a diffeomorphism}

In this section we specialise to the case where $M=X=Y\subset \mathbb{R}^d$ is compact, $T:M\to M$ is a diffeomorphism, and $|\det DT|$ and $h_\mu$ are bounded uniformly above and below.

\subsection{Simplicity of the leading singular value of $\Le$}

Our main result of this section states that the leading singular value of $\Le$ is simple when $\ax=\ay=\mathbf{1}_{B_\epsilon(0)}/\ell(B_\epsilon(0))$;  thus Assumption \ref{ass2} is satisfied.
To set notation, note that one has $\Le^*g(x)=\int_Y k_\epsilon(x,y)g(y)\ d\nu_\epsilon(y)$, so $\mathcal{A}_\epsilon f(y)=\int_X \kappa(x,y)f(x)\ d\mu(x)$, where $\kappa(x,y)=\int_Y k_\epsilon(y,z)k(x,z)\ d\nu(z)$.
Let $\mathcal{A}^q_\epsilon f(y):=\int_X \kappa_q(x,y)\ d\mu(x)$.
The following technical lemma provides sufficient conditions for simplicity of the leading singular value of $\Le$.
\begin{lemma}
\label{simplelemma}
If there exists an integer $q>0$, a $G\in L^1(X,\mu)$ such that $\kappa_q(x,y)\le G(y)$, and a set $A\subset X$ with $\mu(A)>0$ such that $\kappa_q(x,y)>0$ for all $x\in X$ and $y\in A$, then the leading eigenvalue value of $\mathcal{A}_\epsilon$, namely $\sigma_1=1$, is simple.
\end{lemma}
\begin{proof}
By Theorem 5.7.4 \cite{lasota_mackey2} under the hypotheses of the lemma, the Markov operator $\mathcal{A}$ is ``asymptotically stable'' in $L^1(X,\mu)$, meaning that there exists a \emph{unique} $h\in L^1$ (scaled so that $\int_X h\ d\mu=1$) such that $\mathcal{A}h=h$ and $\lim_{k\to\infty} \|\mathcal{A}^kf-h\|_1=0$ for all $f\in L^1$ scaled so that $\int_X f\ d\mu=1$.  Since $\mathcal{A}\mathbf{1}_X=1_X$ by Assumptions \ref{Lasses} (2)-(3), we must have $h=\mathbf{1}_X$ is the unique scaled fixed point in $L^1(X,\mu)$ and as $L^2(X,\mu)\subset L^1(X,\mu)$, $\mathbf{1}_X$ is also the unique scaled fixed point of $\mathcal{A}$ in $L^2(X,\mu)$.
As the eigenvalue 1 is isolated in $L^2(X,\mu)$, it is simple.
Simplicity of the leading singular value of $\mathcal{L}$ follows immediately.
\end{proof}

The ``covering'' hypothesis in Lemma \ref{simplelemma} ($\kappa_q(x,y)>0$ for all $x\in X, y\in A$) can be interpreted dynamically as follows.
Roughly speaking, the kernel $\kappa(x,y)$ is positive if it is possible to be transported from $x\in X$ to an intermediate point $z\in Y_\epsilon$ under forward-time evolution, and then back to $y\in X$ under a dual backward-time evolution.
Thus, $\kappa_q(x,y)>0$ for all $x\in X$ and $y\in A$ if after $q$ forward-backward iterations, there is a positive $\mu$-measure set $A\subset X$ reachable from \emph{all} $x\in X$.
One way to satisfy this is for the kernel $k_\epsilon$ to include some diffusion so that the reachable regions strictly expand with each iteration of $\mathcal{A}_\epsilon$.
This is exactly what happens in a controlled way when we use $\ax=\ay=\mathbf{1}_{B_\epsilon(0)}/\ell(B_\epsilon(0))$.

\begin{proposition}
If  $M=X=Y\subset \mathbb{R}^d$ is compact, $T:M\to M$ is a diffeomorphism, $|\det DT|$ and $h_\mu$ are bounded uniformly above and below, and $\ax=\ay=\mathbf{1}_{B_\epsilon(0)}/\ell(B_\epsilon(0))$, the leading singular value of $\Le$ is simple.
\end{proposition}
\begin{proof}
By Lemma \ref{kboundlemma} (see Appendix), one has that $k_\epsilon$ is bounded.
Note that $$\kappa(x,y)=\int_Y k(y,z)k(x,z)\ d\nu(z)\le \|k\|_\infty\int_Y k(y,z)\ d\nu(z)=\|k\|_\infty.$$
Further,
\begin{eqnarray*}
\kappa_q(x,y)&=&\underbrace{\int_X\int_X\cdots\int_X}_{q-1{\rm\ times}} \kappa(x,x_1)\cdots\kappa(x_{q-2},x_{q-1})\kappa(x_{q-1},y)\ d\mu(x_1)\cdots d\mu(x_{q-1})\\
&\le& \int_X\int_X\cdots\int_X \|k\|_\infty^{q-1}\ d\mu(x_1)\cdots d\mu(x_{q-1})=\|k\|_\infty^{q-1}.
\end{eqnarray*}
Thus, in the hypotheses of Lemma \ref{simplelemma} we may take $G(y)\equiv \|k\|_\infty^{q-1}\in L^1(X,\mu)$.
By Lemma \ref{coverlemma} (see Appendix) the covering hypothesis of Lemma \ref{simplelemma} is satisfied;  the result follows by Lemma \ref{simplelemma}.
\end{proof}

\subsection{Regularity of singular vectors of $\Le$}
\label{regularsection}

A standard heuristic for obtaining partitions from functions is to threshold on level sets;  such an approach has been used in many previous applications of transfer operator methods to determine almost-invariant and metastable sets for autonomous or time-independent dynamical systems \cite{FD03,froyland_05}.
By employing this heuristic, forming eg. $X_1=\{x\in M: f(x)<c\}$ where $f$ is a sub-dominant singular vector of $\mathcal{L}_\epsilon$, and $c\in\mathbb{R}$ is some threshold, if $f$ has some regularity, this places some limitations on the geometrical form of $X_1$.

We derive explicit expressions for the Lipschitz and H\"older constants for the eigenfunctions of $\mathcal{L}_\epsilon^*\mathcal{L}_\epsilon$ and $\mathcal{L}_\epsilon\mathcal{L}_\epsilon^*$, showing how these constants vary as a function of the perturbation parameter $\epsilon$.
%
We begin with two lemmas that demonstrate the regularity of $\Dx f$, for $f\in L^2(X,\mu)$.
Completely analogous results hold for the regularity of $\Dy g$, $g\in L^2(Y'_\epsilon,\nu'_\epsilon)$.
In order to obtain a Lipschitz bound on $\Dx f$, one requires $f$ to be bounded;  on the other hand, the H\"older bound is in terms of the $L^2$-norm of $f$, which is better suited to our setup.
In one direction (right singular vectors of $\mathcal{L}_\epsilon$) we can combine the H\"older bound and the Lipschitz bound to provide a better estimate than a direct H\"older bound.

We remark that our diffusion kernels $\ax, \ay$ need \emph{not} be Lipschitz nor H\"older.
Some prior work has considered the regularity of the range of $\mathcal{P}$ followed by a smoothing operator, denoted by $\mathcal{D}_\epsilon$, in either a $C^0$ or $L^1$ setting:  this includes Zeeman, Lemma 5 \cite{zeeman}, who discusses the equicontinuity of the image of the unit sphere in $C^0$ and $\alpha_\epsilon(x)=n\exp(-\|x\|^2)/2\epsilon$ is smooth, and Junge, Prop. 3.1 \cite{jungethesis}, who bounds the Lipschitz constant of $\mathcal{D}_\epsilon\mathcal{P} f$ in terms of the $L^1$-norm of $f$ using a Lipschitz kernel $\alpha_\epsilon$ (the bound is $L_\alpha\|f\|_{L^1(\ell)}/\epsilon^{1+d}$, where $L_\alpha$ is the Lipschitz constant of a fixed kernel $\alpha$ that generates the family $\alpha_\epsilon$).
The finite-time aspect of our approach (as opposed to the effectively ``asymptotic'' aspect of fixed points of $\mathcal{D}_\epsilon\mathcal{P}$ in \cite{zeeman,jungethesis}) necessitates the use of the operator $\mathcal{D}_\epsilon\mathcal{P}\mathcal{D}_\epsilon$ in order for the singular vectors at both the initial and final times to contain meaningful dynamical information.

As most of the proofs of the following results are technical, we have deferred them to the Appendix.

\begin{lemma}
\label{lipschitzlemma}
Let $f\in L^\infty(X,\ell)$, where $X\subset \mathbb{R}^d$ is compact, $1\le d\le 3$.
Let $\alpha_\epsilon=\mathbf{1}_{B_\epsilon}/\ell(B_\epsilon)$.
Then $\Dx f$ is globally Lipschitz on $X_\epsilon$ with Lipschitz constant bounded above by $C_L(d)\|f\|_{L^\infty}/\epsilon$, where $C_L(d)=1, 4/\pi, 3/2,$ for dimensions $1, 2, 3$, respectively.
\end{lemma}

\begin{lemma}
\label{holderlemma}
Let $f\in L^2(X,\ell)$, where $X\subset \mathbb{R}^d$, $1\le d\le 3$.
Let $\alpha_\epsilon=\mathbf{1}_{B_\epsilon}/\ell(B_\epsilon)$.
Then $\Dx f$ is globally H\"older on $X_\epsilon$:
\begin{equation}
\label{holdereqn}
\left|\Dx f(x)-\Dx f(y)\right|\le \|f\|_{L^2(\ell)}\frac{C(d)}{\epsilon^{(1+d)/2}}\|x-y\|^{1/2},\quad\mbox{ for all $x,y\in X_\epsilon$,}
\end{equation}
where $C(d)=1/\sqrt{2}, \sqrt{2/\pi}, 3/(2\sqrt{2\pi}),$ for dimensions $1, 2, 3$, respectively.
\end{lemma}



\begin{proposition}
\label{Lholderlemma}
\quad
\begin{enumerate}
\item If $f\in L^2(X,\mu)$ with $\int f\ d\mu=0$, then $\mathcal{L}_\epsilon f$ is H\"older and
$$|\mathcal{L}_\epsilon f(x)-\mathcal{L}_\epsilon f(y)|\le \|f\|_{L^2(\mu)}C(d)\cdot C'_H\cdot (1/\epsilon^{d+1})\|x-y\|^{1/2},\quad\mbox{ for all $x,y\in Y_\epsilon$,}$$
in dimensions $d=1,2,3$, where $0<C'_H<\infty$ is a constant depending on properties of $T$ and $h_\mu$.
\item If $f$ is bounded with $\int f\ d\mu=0$, then $\mathcal{L}_\epsilon f$ is Lipschitz
and
$$|\mathcal{L}_\epsilon f(x)-\mathcal{L}_\epsilon f(y)|\le \|f\|_{\infty}C_L(d)\cdot C'_L\cdot (1/\epsilon^2)\|x-y\|,\quad\mbox{ for all $x,y\in Y_\epsilon$,}$$
in dimensions $d=1,2,3$, where $0<C'_L<\infty$ is a constant depending on properties of $T$, $h_\mu$, and $M$.
    \end{enumerate}
    \end{proposition}
The explicit form of $C'_H$ and $C'_L$ are given in the proofs in the appendix.
\begin{proposition}
\label{Lstarholderlemma}
If $g\in L^2(Y,\nu_\epsilon)$
then $\mathcal{L}^*_\epsilon g$ is H\"older and
$$|\mathcal{L}_\epsilon^* g(x)-\mathcal{L}_\epsilon^* g(y)|\le \|g\|_{L^2(\nu_\epsilon)}\frac{1}{\sqrt{A}}\frac{C(d)}{\epsilon^{(1+d)/2}}\|x-y\|^{1/2},\quad\mbox{ for all $x,y\in X_\epsilon$,}$$
in dimensions $d=1,2,3$, where $A=\min_{x\in X_\epsilon} |\det DT(x)|$.
\end{proposition}

As the random perturbations or noise of amplitude $\epsilon$ is increased, Propositions \ref{Lholderlemma} and \ref{Lstarholderlemma} show that the images of $L^2$ functions under $\Le$ and $\Le^*$ become more regular in a H\"older (or Lipschitz) sense.
This information can be used to imply similar regularity results for the left and right singular vectors of $\Le$.
As we anticipate that the optimal $\psi_{X_1,X_2}, \psi_{Y_1,Y_2}$ in the set-based problem (S) (section 3.1) will have coefficients of $O(1)$, we are interested in the minimum spatial distance in phase space that can be traversed by an $O(1)$ difference in value of the singular vectors of $\Le$.
Lower bounds on the $\epsilon$-scaling of these distances are the content of the following corollary.

\begin{corollary}
\label{Lwidthcor}
\quad
\begin{enumerate}
\item Let $f\in L^2(\mu)$ be a subdominant left singular vector of $\mathcal{L}_\epsilon$ (an eigenfunction of $\mathcal{L}^*_\epsilon\mathcal{L}_\epsilon$ corresponding to an eigenvalue less than 1), normalised so that $\|f\|_{L^2(\mu)}=1$.
     An ``$O(1)$ feature'' has width of at least order $\epsilon^{d+1}$.
\item Let $g\in L^2(\nu_\epsilon)$ be a subdominant right singular vector of $\mathcal{L}_\epsilon$ (an eigenfunction of $\mathcal{L}_\epsilon\mathcal{L}^*_\epsilon$ corresponding to an eigenvalue less than 1), normalised so that $\|g\|_{L^2(\nu_\epsilon)}=1$.
     An ``$O(1)$ feature'' has width of at least order $\epsilon^{(5+d)/2}$.
\end{enumerate}
\end{corollary}
\begin{proof}
\quad
\begin{enumerate}
\item A subdominant normalised singular vector $f\in L^2(X,\mu)$ arises as an eigenvector of $\mathcal{L}^*_\epsilon\mathcal{L}_\epsilon$, and in particular $f=\mathcal{L}_\epsilon^* g$ for some $g\in L^2(Y,\nu_\epsilon)$ with $\int g\ d\nu_\epsilon=0$ and $\|g\|_{L^2(\nu_\epsilon)}=1$.
By Proposition \ref{Lstarholderlemma} we see that the H\"older constant of $f=\mathcal{L}_\epsilon^* g$ is $O(1/\epsilon^{(d+1)/2})$.
Thus, if along a given direction, $g$ increases from zero to $O(1)$ and decreases to zero again, the minimal distance required is $O((\epsilon^{(d+1)/2})^2)=O(\epsilon^{d+1})$.
In detail, if $1\le H_{1/2}(f)\cdot |x-y|^{1/2}$ then
\begin{equation*}
|x-y|\ge 1/H_{1/2}(f)^2=A/C(d)^2\epsilon^{d+1}.
\end{equation*}
\item
A subdominant normalised singular vector $g\in L^2(Y_\epsilon,\nu_\epsilon)$ arises as an eigenvector of $\mathcal{L}_\epsilon\mathcal{L}^*_\epsilon$.
We begin with a $g\in L^2(\nu_\epsilon)$ and apply Proposition \ref{Lstarholderlemma} to obtain $\|\mathcal{L}^*_\epsilon g\|_\infty\le \frac{C(d)}{\epsilon^{(1+d)/2}}\frac{1}{\sqrt{A}}\cdot\diam(M)^{1/2}$ using the fact that $\int g\ d\nu=0$.
We now apply Proposition \ref{Lholderlemma} (2) to obtain
\begin{equation*}
L(f)\le \frac{UB\diam(M)^{1/2}C(d)C_L(d)}{A^{3/2}L\epsilon^{(3+d)/2}}+\frac{U^2B^2\diam(M)^{3/2}C(d)C_L(d)^2}{A^{5/2}L^2 \epsilon^{(5+d)/2}}
\end{equation*}
Now, in order to have an O(1) feature we require $1\le L(f)\cdot |x-y|$ or $|x-y|\ge 1/L(f)$.
Since we have an upper bound for $L(f)$, the width of an $O(1)$ feature must be at least $O(\epsilon^{(5+d)/2})$.
\end{enumerate}
\end{proof}

\begin{remark}
Note that if we use only Proposition \ref{Lholderlemma} (1) in the proof of Corollary \ref{Lwidthcor} (2), we would obtain $O(\epsilon^{2d+2})$, which is worse than the estimate in the Corollary.
\end{remark}

\subsection{Scaling of $\sigma_{2,\epsilon}$ with $\epsilon$}

We now demonstrate a lower bound on $\sigma_{2,\epsilon}$, for small $\epsilon$, where $\sigma_{2,\epsilon}$ is the second singular value of $\mathcal{L}_\epsilon$.
A bound similar is spirit to Corollary \ref{corevalbound} was developed in the autonomous two-dimensional area-preserving setting \cite{jungemarsdenmezic}.
Related numerical studies of advection-diffusion PDEs include \cite{cerbellietal,anderson2012}. 
We begin by choosing some $\epsilon^*>0$ and a fixed partition $\{Y_{k,\epsilon^*}\}_{k=1,2}$ of $Y_{\epsilon^*}$.
The partition $\{Y_{k,\epsilon^*}\}_{k=1,2}$ induces compatible partitions for $Y_\epsilon$, $0\le \epsilon<\epsilon^*$, namely, $\{Y_{k,\epsilon}\}_{k=1,2}$, where $Y_{k,\epsilon}:=Y_{k,\epsilon^*}\cap Y_\epsilon$, the restriction of $Y_{k,\epsilon^*}$ to $Y_\epsilon$.
In what follows, it will be useful to consider the ``$\epsilon$-interior'' of a set $A$, denoted $\hat{A}:=\{x\in A: B_\epsilon(x)\subset A\}$.

\begin{lemma}
\label{lowerboundlemma}
Let $T:X\to Y_0$ be non-singular and suppose that $\mu$ is supported on $X$ and absolutely continuous.
One has
\begin{equation}
\label{sigmabound2}
\frac{\mu(\widehat{T^{-1}\hat{Y}_{1,\epsilon}})}{\mu(T^{-1}Y_{1,\epsilon})}+\frac{\mu(\widehat{T^{-1}\hat{Y}_{2,\epsilon}})}{\mu(T^{-1}Y_{2,\epsilon})}\le \frac{\langle \mathcal{L}_\epsilon\mathbf{1}_{X_{1,\epsilon}},\mathbf{1}_{Y_{1,\epsilon}}\rangle_{\nu_\epsilon}}{\mu(X_1)}+\frac{\langle \mathcal{L}_\epsilon\mathbf{1}_{X_{2,\epsilon}},\mathbf{1}_{Y_{2,\epsilon}}\rangle_{\nu_\epsilon}}{\mu(X_2)}\le 1+\sigma_{2,\epsilon},
\end{equation}
for all $0<\epsilon<\epsilon^*$.
\end{lemma}
\begin{proof}
First note that by Theorem \ref{boundlemma}, one has \begin{equation}
\label{sigmabound}
\sigma_2\ge\frac{\langle\mathcal{L}_\epsilon\mathbf{1}_{X_1},\mathbf{1}_{Y_{1,\epsilon}}\rangle_{\nu_\epsilon}}{\mu(X_1)}+\frac{\langle\mathcal{L}_\epsilon\mathbf{1}_{X_2},\mathbf{1}_{Y_{2,\epsilon}}\rangle_{\nu_\epsilon}}{\mu(X_2)}-1,
\end{equation}
for any partition $\{X_1,X_2\}$ of $X$ and $\{Y_{1,\epsilon},Y_{2,\epsilon}\}$ of $Y_\epsilon$.
To partition $X$ we choose $X_{k,\epsilon}:=T^{-1}Y_{k,\epsilon}\cap X$;  one may check that in fact $X_k:=X_{k,0}=X_{k,\epsilon}$ for all $0\le\epsilon\le\epsilon^*$ and that $\{X_k\}_{k=1,2}$ partitions $X$.

In what follows, we will use the following two inequalities:
For $W\subset Y_\epsilon$,
\begin{eqnarray*}
\Dy^*\mathbf{1}_{W}(x)&=&(1/\ell(B_\epsilon(0))\int_{Y_\epsilon} \mathbf{1}_{B_\epsilon(x)}(y)\mathbf{1}_{W}(y)\ dy\\
&=&\frac{\ell(Y_\epsilon\cap B_\epsilon(x)\cap W)}{\ell(B_\epsilon(x))}\\
&\ge& \mathbf{1}_{\hat{W}}(x),
\end{eqnarray*}
and for $V\subset X_\epsilon$,
\begin{eqnarray*}
\Dx^*\mathbf{1}_{V}(x)&=&(1/\ell(B_\epsilon(0))\int_{X_\epsilon} \mathbf{1}_{B_\epsilon(x)}(y)\mathbf{1}_{V}(y)\ dy\\
&=&\frac{\ell(X_\epsilon\cap B_\epsilon(x)\cap V)}{\ell(B_\epsilon(x))}\\
&\ge& \mathbf{1}_{\hat{V}}(x).
\end{eqnarray*}

Now, for $k=1$ (and identically for $k=2$) and $0<\epsilon\le\epsilon^*$ we have
\begin{eqnarray*}
\langle \mathcal{L}_\epsilon\mathbf{1}_{X_{1,\epsilon}},\mathbf{1}_{Y_{1,\epsilon}}\rangle_{\nu_\epsilon}&=&
\langle \mathcal{L}_\epsilon\mathbf{1}_{T^{-1}Y_{1,\epsilon}\cap X},\mathbf{1}_{Y_{1,\epsilon}}\rangle_{\nu_\epsilon}\\
&=&\langle \mathbf{1}_{T^{-1}Y_{1,\epsilon}}\cdot\mathbf{1}_{X},\mathcal{L}_\epsilon^*\mathbf{1}_{Y_{1,\epsilon}}\rangle_{\mu}\\
&=&\langle \mathbf{1}_{T^{-1}Y_{1,\epsilon}},\Dx^*\mathcal{K}\Dy^*\mathbf{1}_{Y_{1,\epsilon}}\rangle_{\mu}\\
&\ge&\langle \mathbf{1}_{T^{-1}Y_{1,\epsilon}},\Dx^*\mathcal{K}\mathbf{1}_{\hat{Y}_{1,\epsilon}}\rangle_{\mu}\\
&=&\langle \mathbf{1}_{T^{-1}Y_{1,\epsilon}},\Dx^*\mathbf{1}_{T^{-1}\hat{Y}_{1,\epsilon}}\rangle_{\mu}\quad\mbox{noting $T^{-1}\hat{Y}_{1,\epsilon}\subset X_\epsilon$} \\
&\ge&\langle \mathbf{1}_{T^{-1}Y_{1,\epsilon}},\mathbf{1}_{\widehat{T^{-1}\hat{Y}_{1,\epsilon}}}\rangle_{\mu}\\
&=&\mu(\widehat{T^{-1}\hat{Y}_{1,\epsilon}}).
\end{eqnarray*}
Since $\mu(X_1)\le \mu(T^{-1}Y_{1,\epsilon})$ we have
\begin{equation*}
\frac{\langle\mathcal{L}_\epsilon\mathbf{1}_{X_1},\mathbf{1}_{Y_{1,\epsilon}}\rangle_{\nu_\epsilon}}{\mu(X_1)}\ge
\frac{\mu(\widehat{T^{-1}\hat{Y}_{1,\epsilon}})}{\mu(T^{-1}Y_{1,\epsilon})},
\end{equation*}
and similarly for $X_2, Y_{2,\epsilon}$, and the result follows.
\end{proof}


\begin{lemma}
\label{hatlowerbound}
 Using the notation of Lemma \ref{lowerboundlemma}, suppose that $Y_0$ is a smooth manifold with smooth boundary, $T^{-1}$ is smooth, and $\mu$ is absolutely continuous with density bounded above and below. Given some $\epsilon^*>0$ one may choose $Y_{1,\epsilon^*},Y_{2,\epsilon^*}$ so that there exists $0<C<\infty$ with
\begin{equation}
\label{sigmabound4}
2-C\epsilon\le \frac{\mu(\widehat{T^{-1}\hat{Y}_{1,\epsilon}})}{\mu(T^{-1}Y_{1,\epsilon})}+\frac{\mu(\widehat{T^{-1}\hat{Y}_{2,\epsilon}})}{\mu(T^{-1}Y_{2,\epsilon})}
\end{equation}
for all $0<\epsilon<\epsilon^*$.
\end{lemma}
\begin{proof}
In order to achieve a tight lower bound for $\sigma_{2,\epsilon}$, we should choose $Y_{1,\epsilon^*}, Y_{2,\epsilon^*}$ judiciously.
Under the hypotheses, $Y_{\epsilon^*}$ has smooth boundary.
Choose a partition $\{Y_{k,\epsilon^*}\}_{k=1,2}$ of two simply connected sets with non-empty interior so that each element has smooth boundary, and so that $\mu(T^{-1}Y_{k,\epsilon^*})=1/2$.
Now for all $0<\epsilon<\epsilon^*$, for small enough $\epsilon$, for each $k=1,2$, one has $\mu(\widehat{T^{-1}\hat{Y}_{k,\epsilon}})/\mu(T^{-1}Y_{k,\epsilon})\ge 1-C\epsilon$, where the constant $C$ depends on $\mu$, $T$, and the curvature of the boundaries of the $Y_{k,\epsilon^*}$.
\end{proof}
\begin{remark}
Note that in order to optimise (minimise) the constant in the $C$ above, one chooses $Y_{k,\epsilon^*}$, $k=1,2$ so that the co-dimension 1 volume of the shared boundary of $Y_{1,\epsilon^*}$ and $Y_{2,\epsilon^*}$ is small \emph{and} the co-dimension 1 volume of the shared boundary of $T^{-1}Y_{1,\epsilon^*}$ and $T^{-1}Y_{2,\epsilon^*}$ is small.
We see here the core of the reason for the ``double'' diffusion (pre- and post- $T$-dynamics) in the definition of $\mathcal{L}_\epsilon$.
If we defined $\mathcal{L}_\epsilon$ as $\Dy\mathcal{P}$ (resp.\ $\mathcal{P}\Dx$), then we would choose $Y_{k,\epsilon^*}$ with small shared boundary (resp.\ choose $T^{-1}Y_{k,\epsilon^*}$ with small shared boundary), and not care about the boundaries of $T^{-1}Y_{k,\epsilon^*}$ (resp.\ $Y_{k,\epsilon^*}$).
By defining $\mathcal{L}_\epsilon=\Dy\mathcal{P}\Dx$ we require small shared boundaries for \emph{both} initial and final time partitions, and guide the 2nd singular vectors (and the subsequently constructed coherent sets) toward having smooth boundaries at initial \emph{and} final times.
\end{remark}

\begin{corollary}
\label{corevalbound}
In the setting of Lemma \ref{lowerboundlemma} and under the hypotheses of Lemma \ref{hatlowerbound}, one has $1-\sigma_{2,\epsilon}\le C\epsilon$ for $0<\epsilon<\epsilon^*$.
\end{corollary}

\section{Numerical Example}
\label{numericalsect}
We now numerically investigate the constructions of the previous sections.
Our numerical example is a quasi-periodically forced flow system representing an idealized stratospheric flow in the northern or southern hemisphere (see \cite{rypina_etal_07}), defined by
\begin{eqnarray*}
\frac{dx}{dt}&=&-\frac{\partial \Psi}{\partial y} \\
\frac{dy}{dt}&=&\frac{\partial \Psi}{\partial x}
\end{eqnarray*}
with streamfunction
\begin{eqnarray*}
\Psi(x,y,t) &=& c_3y -U_0 L \tanh(y/L)+A_3 U_0 L\sech^2(y/L)\cos(k_1 x) \\
&+& A_2 U_0 L\sech^2(y/L)\cos(k_2 x-s_2 t)+ A_1 U_0 L \sech^2(y/L)\cos(k_1 x -s_1 t).
\end{eqnarray*}

We use the parameter values as in \cite{froyland_padberg_12}, i.e.\  $U_0=5.41$, $A_1=0.075$, $A_2=0.4$, $A_3=0.2$, $L=1.770$,  $c_2/U_0=0.205$,$c_3/U_0=0.7$, $k_1=2/r_e$, $k_2=4/r_e$, $k_3=6/r_e$ where $r_e=6.371$ as well as
$s_2=k_2(c_2-c_3)$, $s_1=s_2(1+\sqrt{5})/2$, where we have dropped the physical units for brevity.
We seek coherent sets for the finite-time duration from $t=10$ to $t=20$.
As the flow is area-preserving, we set $h_\mu$ to be constant.
Rypina \emph{et al.} \cite{rypina_etal_07}
show that there is a time-varying jet core oscillating
in a band around $y=0$. The parameters studied
in \cite{rypina_etal_07} were chosen so that the jet core formed a complete
transport barrier between the two Rossby wave regimes
above and below it. In \cite{FSM10} some of the parameters were modified to
remove the jet core band, allowing a small amount of transport between the
two Rossby wave regimes.
Thus we expect that there are two coherent sets;  one above the removed jet core, and one below it.
We demonstrate that we can numerically find these coherent sets, and illustrate the effect of diffusion.

In order to carry out a numerical investigation we require a finite-rank approximation of $\mathcal{L}_\epsilon$.
To construct such an approximation, we use the numerical method of \cite{FSM10}, whereby the domains $X$ and $Y_\epsilon$  are partitioned into boxes and an estimate of $\mathcal{P}_\epsilon$, and then $\mathcal{L}_\epsilon$ is obtained.
We carry out two experiments at differing $\epsilon$ values.
\subsection{Pure advection with implicit numerical diffusion}
Firstly, we directly use the technique from \cite{FSM10}.
We partition $[0,20]\times[-2.5,2.5]$ into $2^{15}$ boxes, leading to boxes of radius $0.0391\times 0.0195$.
We obtain finite rank estimates of $\mathcal{P}$ and $\mathcal{L}_0$, which are $32768\times 39465$ sparse non-negative matrices $P$ and $L$, respectively;  these matrices was produced using 400 test points per box (see \cite{FSM10} for details).
The leading singular value of $L$ is 1, and the second singular value is $\sigma_2=0.9969$.
The discretisation procedure used to estimate $\mathcal{L}_0$ leads to a ``numerical diffusion'' so that one effectively estimates $\mathcal{L}_\epsilon$ with $\epsilon\lessapprox 0.0391$;  this is the reason for the spectral gap appearing in the numerical estimate of $\mathcal{L}_0$.
An estimate of $h_{\nu_\epsilon}$, produced as $\mathbf{1}P$, is shown in Figure 1 (left).
\begin{figure}
  \label{pushforward0}
  \hspace*{-1.5cm}
  \includegraphics[width=10cm]{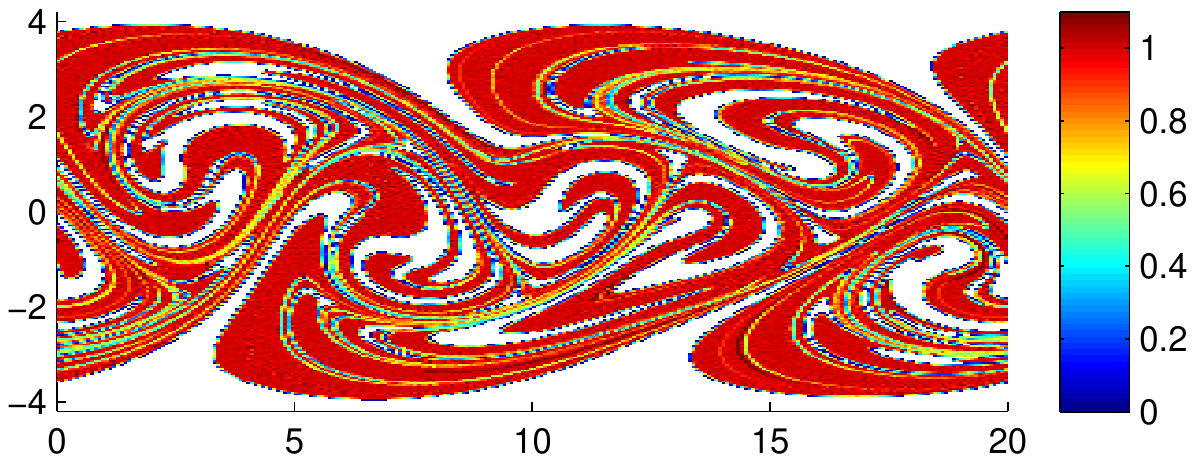}
  \includegraphics[width=10cm]{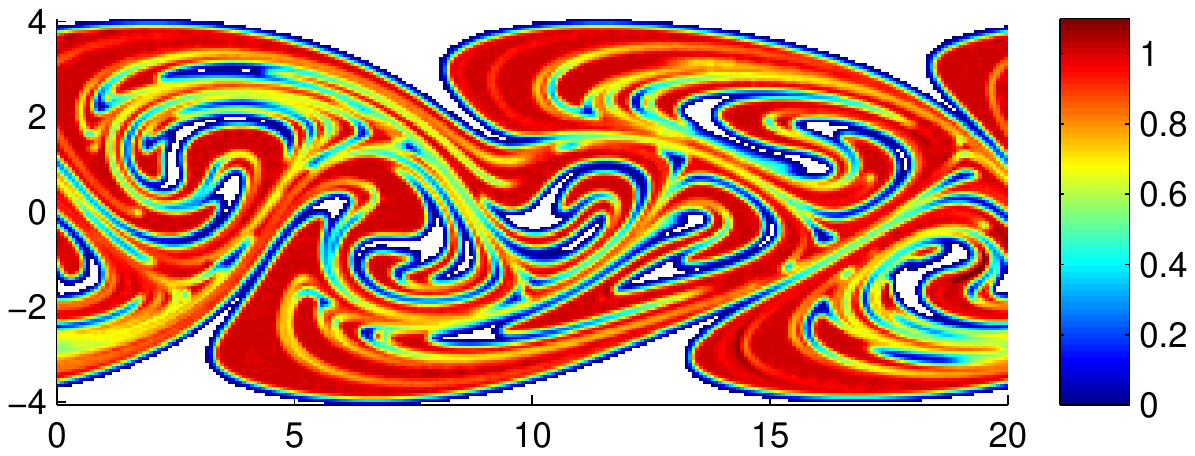}\\
  \caption{Left: Density of pushforward of Lebesgue measure on $[0,20]\times [-2.5,2.5]$ by a discrete form of $\mathcal{P}_\epsilon$, $\epsilon\approx 0.0391$ (a discrete approximation of $h_{\nu_\epsilon}$).  Right:  As for Left, with $\epsilon=0.1$.}
\end{figure}
The second left (resp.\ right) singular vector is shown in the left image of Figure 2 (resp.\ Figure 3).
Note that the value of the singular vectors in Figures 2 and 3 is predominantly in the vicinity of $\pm 1$, and that there is a clear separation into two coherent sets, consistent with the known facts about transport in this system.
\begin{figure}
  \label{leftsingvec}
  \hspace*{-1.5cm}
  \includegraphics[width=10cm]{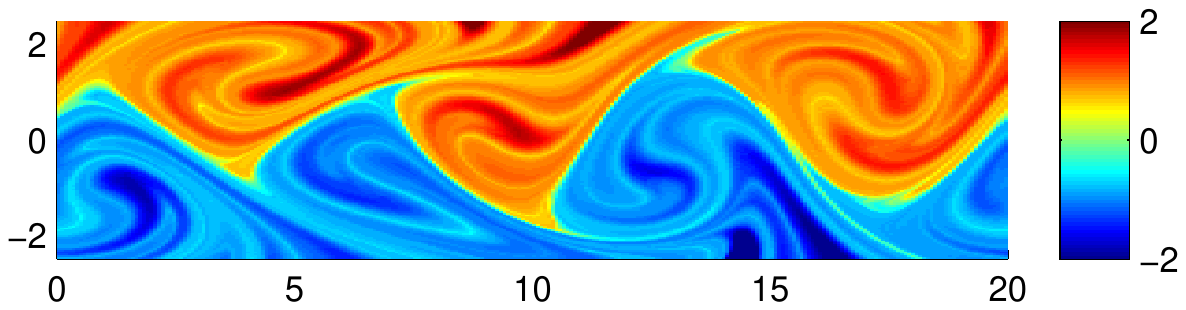}
  \includegraphics[width=10cm]{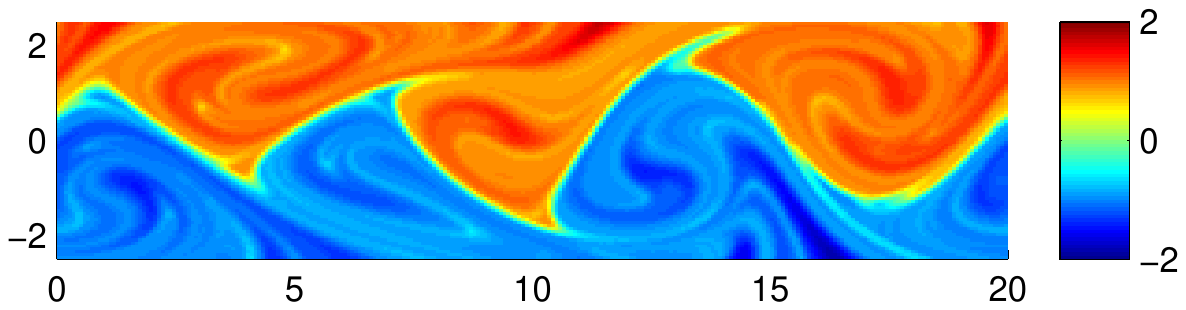}\\
  \caption{Left: Discrete approximation of the second left singular vector of $\mathcal{L}_\epsilon$, $\epsilon\approx 0.0391$.  Right:  As for Left, with $\epsilon=0.1$.}
\end{figure}
\begin{figure}
 \label{rightsingvec} 
     \hspace*{-1.5cm}
    \includegraphics[width=10cm]{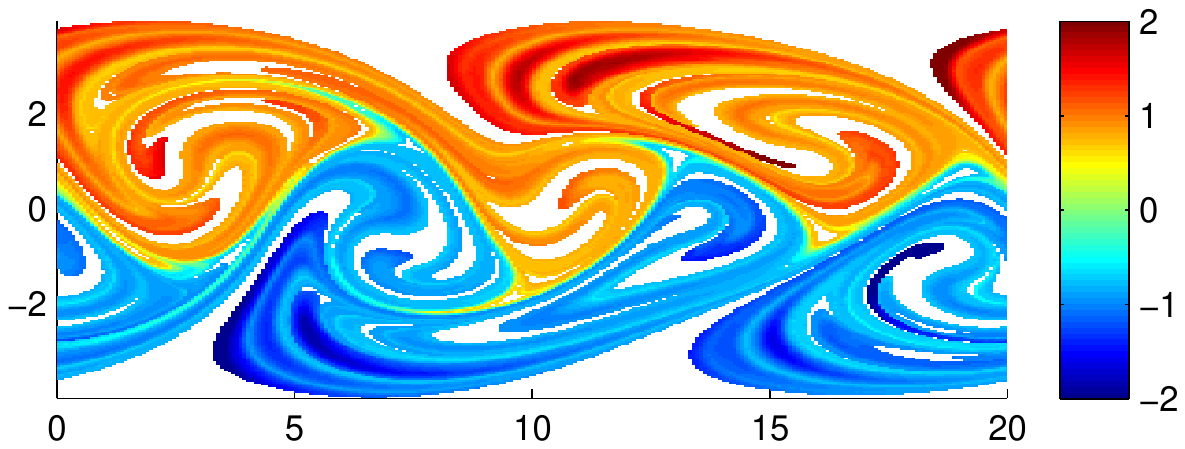}
  \includegraphics[width=10cm]{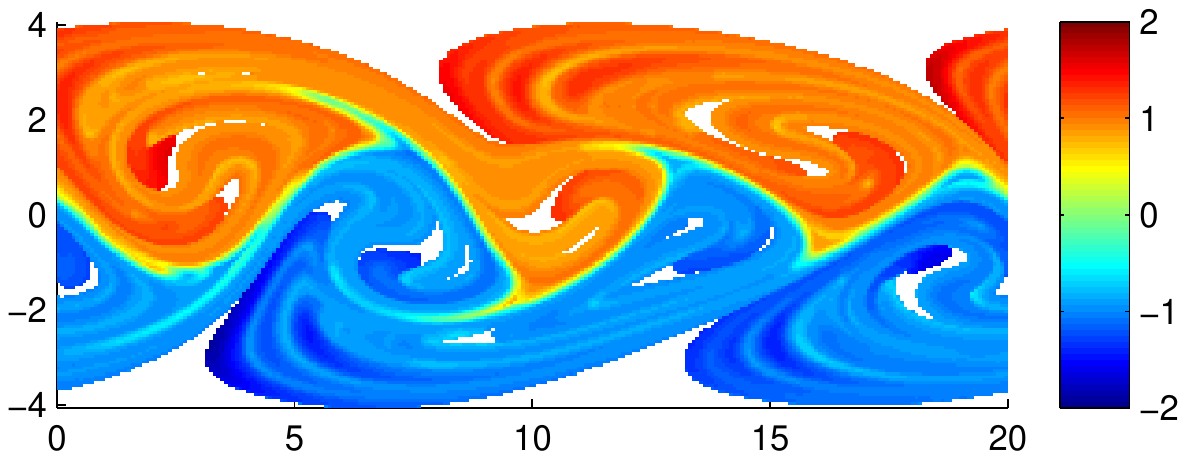}\\
   \caption{Left: Discrete approximation of the second right singular vector of $\mathcal{L}_\epsilon$, $\epsilon\approx 0.0391$.  Right:  As for Left, with $\epsilon=0.1$.}
\end{figure}
An optimal level set thresholding of the second left (resp.\ right) singular vector is shown in the left image of Figure 4 (resp.\ Figure 5);  see Remark \ref{remarkN} for details.
\begin{figure}
    \label{leftthresh}
  \hspace*{-1.5cm}  \includegraphics[width=10cm]{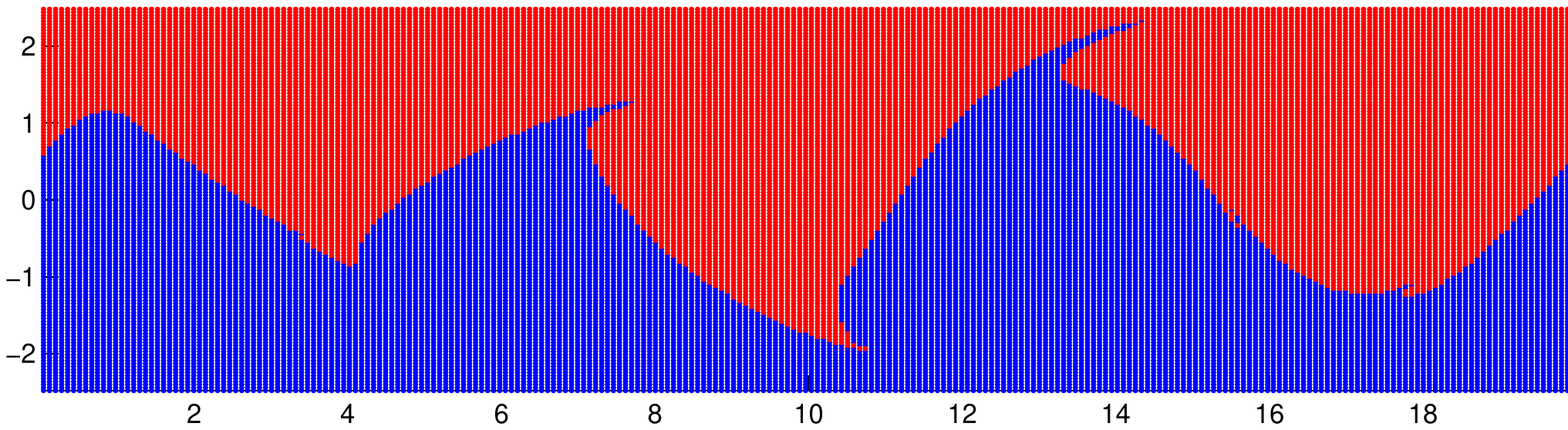}
  \includegraphics[width=10cm]{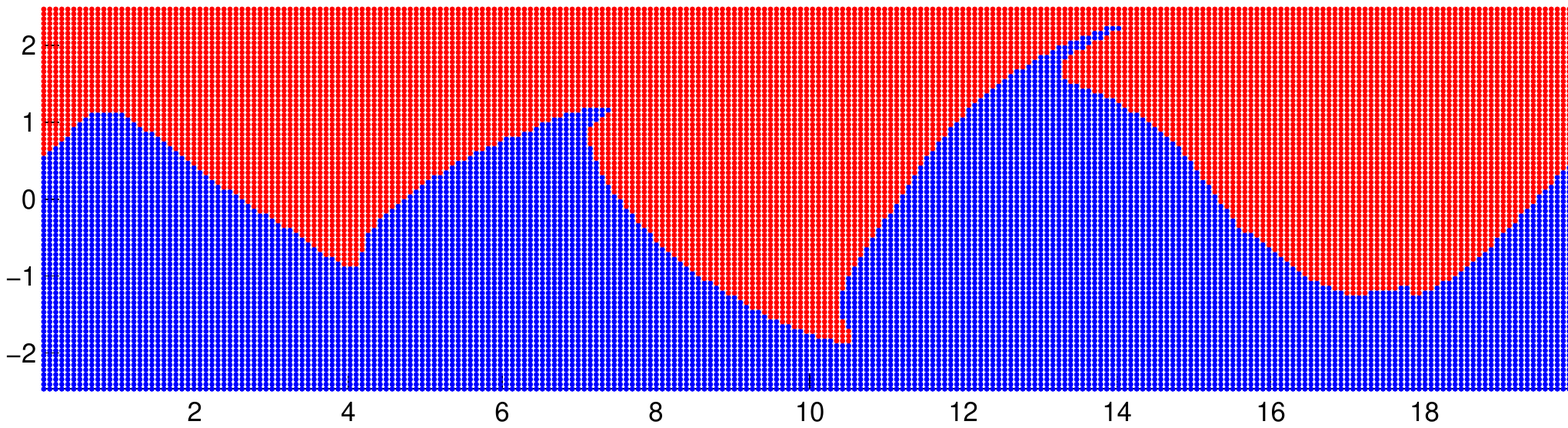}\\
  \caption{Left: Thresholding second left singular vector of $\mathcal{L}_\epsilon$, $\epsilon\approx 0.0391$ to maximise expression (\ref{form1}).  Right:  As for Left, with $\epsilon=0.1$ }
\end{figure}
\begin{figure}
  \label{rightthresh} 
\hspace*{-1.5cm}
      \includegraphics[width=10cm]{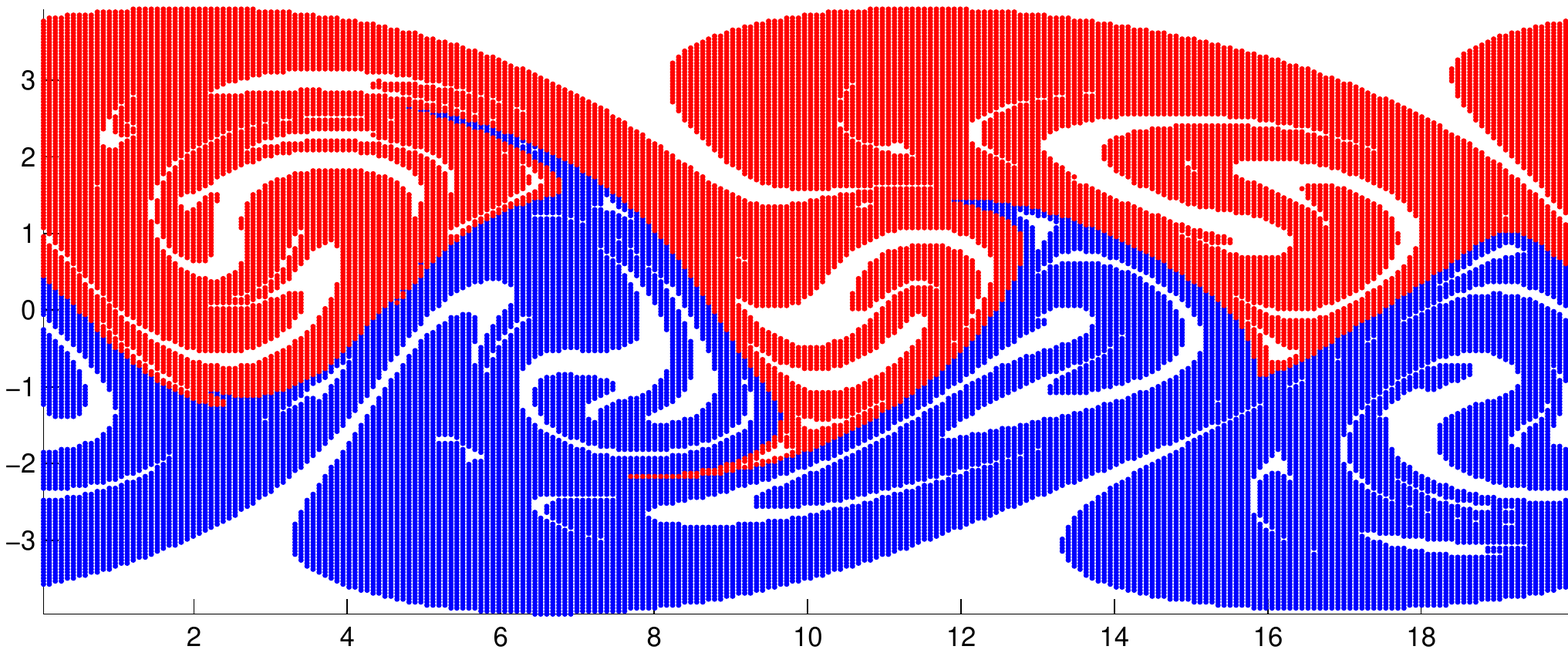}
  \includegraphics[width=10cm]{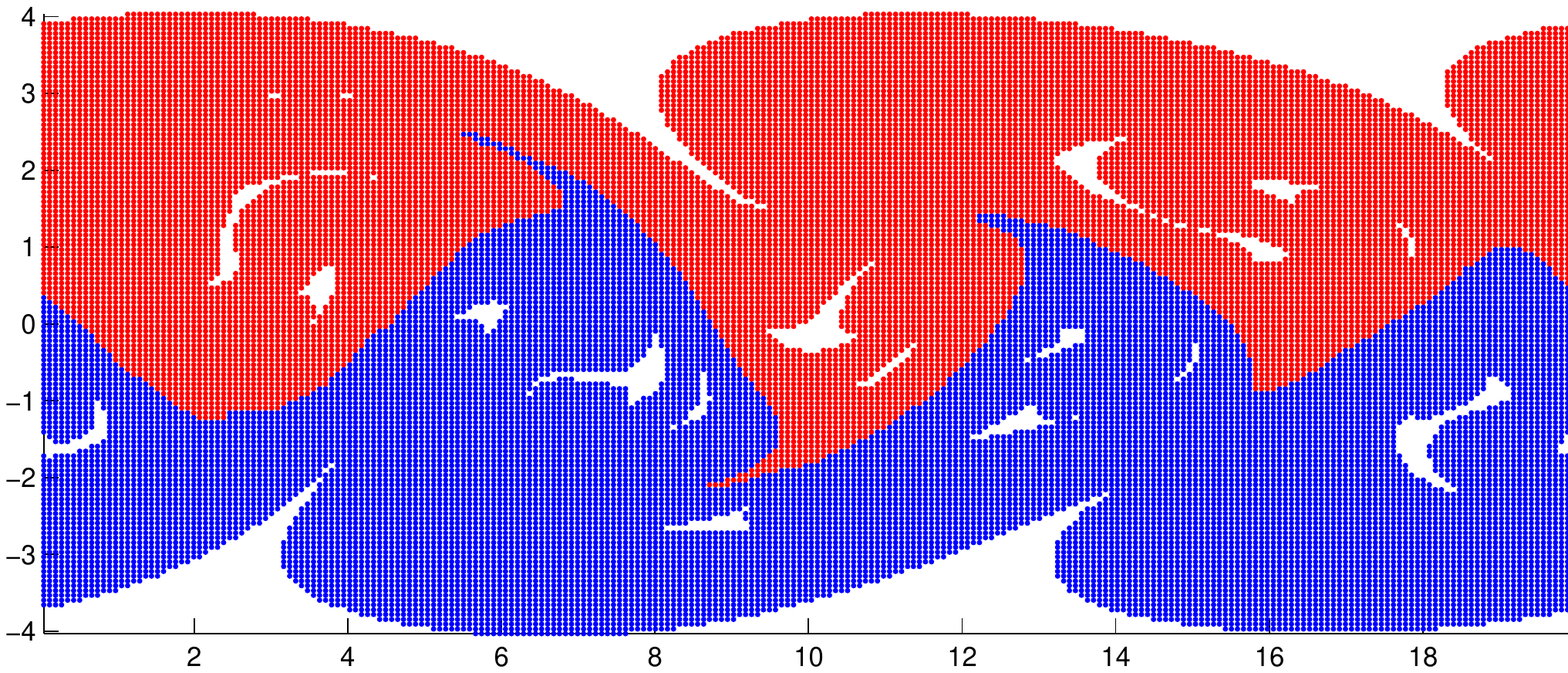}\\
    \caption{Left: Thresholding second right singular vector of $\mathcal{L}_\epsilon$, $\epsilon\approx 0.0391$ to maximise expression (\ref{form1}).  Right:  As for Left, with $\epsilon=0.1$}
\end{figure}
The maximal value in (\ref{form1}) obtained from this level set thresholding procedure was computed to be 1.9854, consistent with (and close to) the upper bound of 1.9969 given by (\ref{boundeqn}).
The $\mu$-measures of the red (resp.\ blue) sets are approximately 0.5064 (resp.\ 0.4936).
\subsection{Advection with explicit diffusion}
Secondly, we explicitly apply a diffusion of radius $\epsilon=0.1$ before and after the action of $\mathcal{P}_\epsilon$.
Numerically, this is achieved by applying a mask of 37 points in an $\epsilon$-ball about each of the 36 test points per box, before and after the action of the deterministic dynamics.
This results in a $20480\times 29071$ matrices $P_\epsilon$ and $L_\epsilon$ with the latter having a leading singular value of 1 and second singular value $\sigma_2=0.9793$.
An estimate of $h_{\nu_\epsilon}$, produced as $\mathbf{1}P_\epsilon$, is shown in Figure 1 (left).
The second left (resp.\ right) singular vector is shown in the right image of Figure 2 (resp.\ Figure 3).
An optimal level set thresholding of the second left (resp.\ right) singular vector is shown in the right image of Figure 4 (resp.\ Figure 5); see Remark \ref{remarkN} for details.
The maximal value in (\ref{form1}) obtained from this level set thresholding procedure was computed to be 1.9544, consistent with the upper bound of 1.9793 given by (\ref{boundeqn}).
The $\mu$-measures of the red (resp.\ blue) sets are approximately 0.5031 (resp.\ 0.4969).

Furthermore, the increasing regularity of the left and right singular vectors with increasing $\epsilon$, as predicted by Propositions \ref{Lholderlemma} and \ref{Lstarholderlemma} is evident in Figures 2 and 3.

%
%

\appendix
\section{Proofs}
\subsection{Boundedness and compactness of $\Le$ and $\Le^*$}
The following Lemma is a straight-forward modification of Prop II.1.6 in Conway \cite{conway}, which we include for completeness.
\begin{lemma}
\label{Lbound}
Let $(X,\mu)$ and $(Y,\nu)$ be measure spaces and suppose $k:X\times Y\to \mathbb{R}$ is measurable, and that
\begin{eqnarray}
\label{int1}
\int_X |k(x,y)|\ d\mu(x)\le c_1&&\quad\mbox{for $\nu$-almost all $y\in Y$},\\
\label{int2}
\int_Y |k(x,y)|\ d\nu(y)\le c_2&&\quad\mbox{for $\mu$-almost all $x\in X$}.
\end{eqnarray}
If $K:L^2(X,\mu)\to L^2(Y,\nu)$ is defined by $Kf(x)=\int_X k(x,y)f(y)\ d\mu(y)$ then $K$ is a bounded linear operator and $\|K\|\le (c_1c_2)^{1/2}$.
\end{lemma}
\begin{proof}
\begin{eqnarray*}
|Kf(y)|&\le&\int_X|k(x,y)||f(x)|\ d\mu(x)\\
&=&\int_X|k(x,y)|^{1/2}|k(x,y)|^{1/2}|f(x)|\ d\mu(x)\\
&\le&\left(\int_X|k(x,y)|\ d\mu(x)\right)^{1/2}\left(\int_X|k(x,y)||f(x)|^2\ d\mu(x)\right)^{1/2}\\
&=&c_1^{1/2}\left(\int_X|k(x,y)||f(x)|^2\ d\mu(x)\right)^{1/2}
\end{eqnarray*}
Thus,
\begin{eqnarray*}
\|Kf\|_{L^2(\nu)}^2&=&\int_Y |Kf(y)|^2\ d\nu(y) \\
&\le&\int_Yc_1\int_X|k(x,y)||f(x)|^2\ d\mu(x)d\nu(y)\\
&=&\int_X|f(x)|^2c_1\int_Y|k(x,y)|\ d\nu(y)d\mu(x)\\
&=&\int_Xc_1c_2|f(x)|^2\ d\mu(x)=c_1c_2\|f\|_{L^2(\mu)}^2
\end{eqnarray*}
\end{proof}
\begin{proof}[Proof of Lemma \ref{compactlemma}]
We follow the proofs of Proposition II.4.7 and Lemma II.4.8 \cite{conway}, generalising from $L^2(X\times X,\mu\times\mu)$ to $L^2(X\times Y,\mu\times\nu)$.
Inner products and norms on $L^2(X,\mu)$ and $L^2(Y,\nu)$ will have subscripts $\mu$ and $\nu$ respectively;  no subscript means an inner product or norm on $L^2(X\times Y,\mu\times\nu)$.
 The assumptions on the integrability of $k$ guarantee that $\mathcal{L}:L^2(X,\mu)\to L^2(Y,\nu)$ and its dual are bounded linear operators: \ $\|\mathcal{L}f\|^2_{\nu}=\int(\int k(x,y)f(x)\ d\mu(x))^2\ d\nu(y)\le \int(\int k(x,y)^2\ d\mu(x)\cdot \int f(x)^2\ d\mu(x))\ d\nu(y)=\|k\|^2\|f\|^2_{\mu}$;  similarly for $\mathcal{L}^*$.

Let $\{e_i\}$ be an orthonormal basis of $L^2(X,\mu)$ and $\{f_j\}$ be an orthonormal basis of $L^2(Y,\nu)$.
Define $\phi_{ij}(x,y)=\bar{e}_i(x)f_j(y)$.
It is easy to check that $\{\phi_{ij}\}$ is an orthonormal set in $L^2(X\times Y,\mu\times\nu)$:
\begin{eqnarray*}
\langle \phi_{ij},\phi_{kl}\rangle&=&\int \bar{e}_i(x)f_j(y)\cdot e_k(x)\bar{f}_l(y)\ d\mu(x)d\nu(y)\\
&=&\int\left(\int\bar{e}_i(x)e_k(x)\ d\mu(x)\right)f_j(y)\bar{f}_l(y)\ d\nu(y)\\
&=&\delta_{ki}\delta_{jl}.
\end{eqnarray*}
One also has
\begin{eqnarray*}
\langle k,\phi_{ij}\rangle&=&\int k(x,y)\bar{\phi}_{ij}(x,y)\ d\mu(x)d\nu(y)\\
&=&\int k(x,y)e_i(x)\bar{f}_j(y)\ d\mu(x)d\nu(y)\\
&=&\langle \mathcal{L}e_i,f_j\rangle_\nu.
\end{eqnarray*}
Therefore $\|k\|^2\ge \sum_{i,j}|\langle k,\phi_{ij}\rangle|^2=\sum_{i,j}|\langle \mathcal{L}e_i,f_j\rangle_\nu|^2.$
Since $k\in L^2(X\times Y,\mu\times\nu)$ there are at most a countable number of $i,j$ such that $\langle k,\phi_{ij}\rangle\neq 0$;  denote these by $\{\psi_{lm}\}$ and note that $\langle k,\phi_{ij}\rangle=0$ unless $\phi_{ij}\in\{\psi_{lm}\}$.
Let $\psi_{lm}(x,y)=\bar{e}_l(x)f_m(y)$, let $P_n:L^2(X,\mu)\to L^2(X,\mu)$ be the orthogonal projection onto $\{e_l:1\le l\le n\}$ and $Q_n:L^2(Y,\nu)\to L^2(Y,\nu)$ be the orthogonal projection onto $\{f_m:1\le m\le n\}$.
Define $\mathcal{L}_n=\mathcal{L}P_n+Q_n\mathcal{L}-Q_n\mathcal{L}P_n$, a finite-rank operator.
We will show that $\|\mathcal{L}-\mathcal{L}_n\|\to 0$ as $n\to\infty$, showing $\mathcal{L}$ is compact.

Let $g\in L^2(X,\mu)$ with $\|g\|_\mu\le 1$, and write $g=\sum_i\alpha_ie_i$.
One has
\begin{eqnarray*}
\|\mathcal{L}g-\mathcal{L}_ng\|_\nu^2&=&\sum_j|\langle\mathcal{L}g-\mathcal{L}_ng,f_j\rangle_\nu|^2\\
&=&\sum_j\left|\sum_i\alpha_i\langle(\mathcal{L}-\mathcal{L}_n)e_i,f_j\rangle_\nu\right|^2\\
&=&\sum_m\left|\sum_l\alpha_l\langle(\mathcal{L}-\mathcal{L}_n)e_l,f_m\rangle_\nu\right|^2\\
&\le&\sum_m\left(\sum_l|\alpha_l|^2\right)\left(\sum_l|\langle(\mathcal{L}-\mathcal{L}_n)e_l,f_m\rangle_\nu|^2\right)\\
&=&\|g\|_\mu^2\sum_m\sum_l\left|\langle\mathcal{L}e_l,f_m\rangle_\nu-\langle\mathcal{L}_ne_l,f_m\rangle_\nu\right|^2\\
&=&\|g\|_\mu^2\sum_m\sum_l\left|\langle\mathcal{L}e_l,f_m\rangle_\nu-\langle\mathcal{L}P_ne_l,f_m\rangle_\nu-\langle Q_n\mathcal{L}e_l,f_m\rangle_\nu+\langle Q_n\mathcal{L}P_ne_l,f_m\rangle_\nu\right|^2\\
&=&\|g\|_\mu^2\sum_{m>n}\sum_{l>n}\left|\langle\mathcal{L}e_l,f_m\rangle_\nu\right|^2\\
&=&\|g\|_\mu^2\sum_{m>n}\sum_{l>n}|\langle k,\phi_{lm}\rangle|^2.
\end{eqnarray*}
The penultimate equality holds since when $l\le n$, $P_ne_l=e_l$ and when $m\le n$, $Q_nf_m=f_m$;  thus if either $l\le n$ or $m\le n$, the entire expression is zero.
Finally, since $\sum_{l,m}|\langle k,\phi_{lm}\rangle|^2\le \|k\|<\infty$, given $\epsilon>0$, choose $n$ large enough so that this sum is less than $\epsilon^2$.
Since $\mathcal{L}$ is compact, so is $\mathcal{L}^*$.
\end{proof}

\subsection{Simplicity of the leading singular value of $\Le$}

For the two lemmas in this section we assume that $X=Y_\epsilon=M\subset \mathbb{R}^d$ is compact, that $T$ is a diffeomorphism, that $\mu$ is absolutely continuous with positive density, and that $\ax=\ay=\mathbf{1}_{B_\epsilon(0)}/\ell(B_\epsilon(0))$.

\begin{lemma}
\label{kboundlemma}
$k_\epsilon(x,y)$ is bounded.
\end{lemma}
\begin{proof}
Recall $$k_\epsilon(x,y)=\frac{\ell(B_\epsilon(x)\cap T^{-1}B_\epsilon(y))}{\int_X \ell(B_\epsilon(x)\cap T^{-1}B_\epsilon(y))\ d\mu(x)}.$$
Clearly the numerator is uniformly bounded above by $\ell(B_\epsilon(0))$;  we now show that the denominator is uniformly bounded below.
One has $\ell(T^{-1}B_\epsilon)\ge \ell(B_\epsilon(0))/\sup_{x\in M}|\det DT(x)|$.
Define $\delta^*=\sup\{\delta: \mbox{ for all }y\in M, \exists z=z(y)\mbox{ such that }B_\delta(z)\subset T^{-1}B_\epsilon(y)\}$.
Since $|\det DT|$ is uniformly bounded above, $\delta^*>0$.
Then $\int \ell(B_\epsilon(x)\cap T^{-1}B_\epsilon(y))\ d\mu(x)\ge D^*\cdot \ell(B_{\delta^*/2}(0))$ for all $y\in M$, where $D^*=\inf_{y\in M} \mu(\{z\in T^{-1}B_\epsilon(y): B_{\delta^*/2}(z)\subset T^{-1}B_\epsilon(y)\}$.
Since $\mu$ is absolutely continuous with positive density, $D^*>0$.
\end{proof}


\begin{lemma}
\label{coverlemma}
There exists a $q>0$ and an $A\subset X$ with $\mu(A)>0$ such that $\kappa_q(x,y)>0$ for $\mu$-a.a. $x\in X$ and $y\in A$.
\end{lemma}
\begin{proof}
Consider the action of $\mathcal{A}_\epsilon=\Le^*\Le$ on a distribution $\delta_x$.  As $\Le(f)=\Dy\circ\P\circ\Dx(f)/\Dy\circ\P\circ\Dx(f\cdot h_\mu)$, the initial application of $\Dx$ produces a function with support $B_\epsilon(x)$ (in fact, $\mathbf{1}_{B_\epsilon(x)}$). The application of $\P$ now creates a function with support $T(B_\epsilon(x))$ (because $|\det DT|$ is uniformly bounded above) and finally the application of $\Dy$ produces a support of $B_\epsilon(T(B_\epsilon(x)))$.
Now applying $\Le^*$ we again apply $\Dx$, then $T^{-1}$, then $\Dy$, producing a function with support $S=B_\epsilon(T^{-1}(B_\epsilon(B_\epsilon(T(B_\epsilon(x))))))$.
Clearly $B_{\epsilon/2}(x)\subset S$.
At each iteration of $\mathcal{A}_\epsilon$ the support expands by at least $\epsilon/2$.  As $X$ is bounded, eventually the support fills $X$ after $q$ iterations for some $q$.
Thus $\kappa_q(x,y)>0$ for $\mu$-a.a. $x\in X$ and $y\in X$.
\end{proof}

\subsection{Regularity of singular vectors of $\Le$}

\begin{proof}[Proof of Lemma \ref{lipschitzlemma}]
One has
\begin{eqnarray}
\nonumber\frac{\left|\Dx f(y+\gamma)-\Dx f(y)\right|}{\|\gamma\|}
&=&\frac{\left|\int_X(\ax(y+\gamma-x)-\ax(y-x))f(x)\ dx\right|}{\|\gamma\|} \\
\nonumber&\le &\|f\|_{L^\infty}\cdot\frac{\int_X|\ax(y+\gamma-x)-\ax(y-x)|\ dx}{\|\gamma\|} \\
\nonumber&= &\|f\|_{L^\infty}\cdot\frac{\int_X|\mathbf{1}_{B_\epsilon(y+\gamma)}-\mathbf{1}_{B_\epsilon(y)}|\ dx}{\ell(B_\epsilon(y))\|\gamma\|} \\
\label{finalestlip}&\le&\|f\|_{L^\infty}\cdot\frac{\ell(B_\epsilon(y+\gamma)\triangle B_\epsilon(y))}{\ell(B_\epsilon(y))\|\gamma\|}.
\end{eqnarray}
We show that $\limsup_{\|\gamma\|\to 0}\left|\Dx f(y+\gamma)-\Dx f(y)\right|/\|\gamma\|\le C\|f\|_{L^\infty}/\epsilon$, for all $x\in X$.
From this, global Lipschitzness follows by Lemma \ref{localgloballipschitz}.
We now detail the computations for dimensions 1, 2, and 3.
The main estimate is for $\ell(B_\epsilon(y+\gamma)\triangle B_\epsilon(y))$ in (\ref{finalestlip}), which only depends on $\epsilon$ and $\gamma$.

\emph{Dimension 1:} $B_\epsilon$ is an interval of length $2\epsilon$ and clearly $\ell(B_\epsilon(y+\gamma)\triangle B_\epsilon(y))=2|\gamma|$ for $0\le |\gamma|\le 2\epsilon$.
Thus for $0\le|\gamma|\le 2\epsilon$, $(\ref{finalestlip})=\|f\|_{L^\infty}2|\gamma|/(2\epsilon|\gamma|)=\|f\|_{L^\infty}/\epsilon$.

\emph{Dimension 2:} $B_\epsilon$ is a disk of radius $\epsilon$.  For $0\le \|\gamma\|\le 2\epsilon$, the symmetric difference area $\ell(B_\epsilon(y+\gamma)\triangle B_\epsilon(y))$ is $A=2(\pi\epsilon^2-(2\epsilon^2\cos^{-1}(\|\gamma\|/2\epsilon)-(1/2)\|\gamma\|\sqrt{4\epsilon^2-\|\gamma\|^2}))$ \cite{weisstein2}.
To first order in $\|\gamma\|$, $\cos^{-1}(\|\gamma\|/2\epsilon)=\pi/2-\|\gamma\|/2\epsilon+O(\|\gamma\|^2)$.
Thus $A=2(\pi\epsilon^2-\pi\epsilon^2+\|\gamma\|\epsilon+(1/2)\|\gamma\|\sqrt{4\epsilon^2-\|\gamma\|^2})+O(\|\gamma\|^2)$ and
$\limsup_{\|\gamma\|\to 0} \|f\|_{L^\infty}A/\|\gamma\|\pi\epsilon^2=\|f\|_{L^\infty}(4/\pi)/\epsilon.$

\emph{Dimension 3:}$B_\epsilon$ is a disk of radius $\epsilon$.  For $0\le \|\gamma\|\le 2\epsilon$, the symmetric difference volume $\ell(B_\epsilon(y+\gamma)\triangle B_\epsilon(y))$ is $V=2(4/3\pi\epsilon^3-(\pi/12(4\epsilon+\|\gamma\|)(2\epsilon-\|\gamma\|)^2))$ \cite{weisstein3}.
To first order in $\|\gamma\|$, $V=2(4/3\pi\epsilon^3-(\pi/12(16\epsilon^3-12\epsilon^2\|\gamma\|)))+O(\|\gamma\|^2)$.
Thus $\lim_{\|\gamma\|\to 0} \|f\|_{L^\infty}V/\|\gamma\|(4/3)\pi\epsilon^3=\|f\|_{L^\infty}2\pi\epsilon^2\|\gamma\|/\|\gamma\|(4/3)\pi\epsilon^3=\|f\|_{L^\infty}(3/2)/\epsilon$.

By Lemma \ref{localgloballipschitz} the constants found above are also global Lipschitz constants.
\end{proof}

\begin{proof}[Proof of Lemma \ref{holderlemma}]
Let $x=y+\gamma$.
One has
\begin{eqnarray}
\nonumber\left|\Dx f(y+\gamma)-\Dx f(y)\right|&=&\left|\int_X(\ax(y+\gamma-x)-\ax(y-x))f(x)\ dx\right| \\
\nonumber&\le &\|f\|_{L^2(\ell)}\cdot\left(\int_X(\ax(y+\gamma-x)-\ax(y-x))^2\ dx\right)^{1/2} \\
\nonumber&= &\|f\|_{L^2(\ell)}\cdot\left(\int_X(\mathbf{1}_{B_\epsilon(y+\gamma)}-\mathbf{1}_{B_\epsilon(y)})^2\ dx\right)^{1/2}/\ell(B_\epsilon(y)) \\
\label{finalestholder}&\le&\|f\|_{L^2(\ell)}\cdot\left(\ell(B_\epsilon(y+\gamma)\triangle B_\epsilon(y))\right)^{1/2}/\ell(B_\epsilon(y))
\end{eqnarray}
We now detail the computations for dimensions 1, 2, and 3.
The main estimate is for $\ell(B_\epsilon(y+\gamma)\triangle B_\epsilon(y))$ in (\ref{finalestholder}), which only depends on $\epsilon$ and $\gamma$.

\emph{Dimension 1:}  $B_\epsilon$ is an interval of length $2\epsilon$ and clearly
$$\ell(B_\epsilon(y+\gamma)\triangle B_\epsilon(y))=\left\{
                                                   \begin{array}{ll}
                                                     2\|\gamma\|, & \hbox{$\|\gamma\|\le 2\epsilon$;} \\
                                                     2\epsilon, & \hbox{$\|\gamma\|>2\epsilon$.}
                                                   \end{array}
                                                 \right..$$
Thus $$\left(\ell(B_\epsilon(y+\gamma)\triangle B_\epsilon(y))\right)^{1/2}/\ell(B_\epsilon(y))=\left\{
                          \begin{array}{ll}
                            \frac{\sqrt{2\|\gamma\|}}{2\epsilon}, & \hbox{$\|\gamma\|\le \epsilon$;} \\
                            \frac{\sqrt{2\epsilon}}{2\epsilon}, & \hbox{$\|\gamma\|>\epsilon$.}
                          \end{array}
                        \right.,$$
and one has $|\Dx f(x)-\Dx f(y)|\le \|f\|_{L^2(\ell)}\cdot\sqrt{\frac{1}{2}}\frac{1}{\epsilon}\|x-y\|^{1/2}$.

\emph{Dimension 2:} $B_\epsilon$ is a disk of radius $\epsilon$.  For $0\le \|\gamma\|\le 2\epsilon$, the symmetric difference area $\ell(B_\epsilon(y+\gamma)\triangle B_\epsilon(y))$ is $A=2(\pi\epsilon^2-(2\epsilon^2\cos^{-1}(\|\gamma\|/2\epsilon)-(1/2)\|\gamma\|\sqrt{4\epsilon^2-\|\gamma\|^2}))$ \cite{weisstein2}.
Thus, $$\ell(B_\epsilon(y+\gamma)\triangle B_\epsilon(y))=\left\{
                                                   \begin{array}{ll}
                                                     2(\pi\epsilon^2-(2\epsilon^2\cos^{-1}(\|\gamma\|/2\epsilon)-(1/2)\|\gamma\|\sqrt{4\epsilon^2-\|\gamma\|^2})), & \hbox{$\|\gamma\|\le 2\epsilon$;} \\
                                                     2\pi\epsilon^2, & \hbox{$\|\gamma\|>2\epsilon$.}
                                                   \end{array}
                                                 \right..$$
\begin{eqnarray*}
\lefteqn{\left(\ell(B_\epsilon(y+\gamma)\triangle B_\epsilon(y))\right)^{1/2}/\ell(B_\epsilon(y))}\\
&=&\left\{
                          \begin{array}{ll}
                            \sqrt{2(\pi\epsilon^2-(2\epsilon^2\cos^{-1}(\|\gamma\|/2\epsilon)-(1/2)\|\gamma\|\sqrt{4\epsilon^2-\|\gamma\|^2}))}/(\pi\epsilon^2), & \hbox{$\|\gamma\|\le 2\epsilon$;} \\
                            \sqrt{2\pi\epsilon^2}/(\pi\epsilon^2), & \hbox{$\|\gamma\|>2\epsilon$.}
                          \end{array}
                        \right.,
                        \end{eqnarray*}
and one has $$|\Dx f(x)-\Dx f(y)|\le \|f\|_{L^2(\ell)}\cdot\sqrt{2/\pi}\frac{1}{\epsilon^{3/2}}\|x-y\|^{1/2}$$ for all $x,y\in X$.

\emph{Dimension 3:}$B_\epsilon$ is a disk of radius $\epsilon$.  For $0\le \|\gamma\|\le 2\epsilon$, the symmetric difference volume $\ell(B_\epsilon(y+\gamma)\triangle B_\epsilon(y))$ is \cite{weisstein3}
\begin{eqnarray*}
V&=&2(4/3\pi\epsilon^3-(\pi/12(4\epsilon+\|\gamma\|)(2\epsilon-\|\gamma\|)^2))\\
&=&8\pi\epsilon^3/3-((\pi/6)(4\epsilon+\|\gamma\|)(4\epsilon^2-4\|\gamma\|\epsilon+\|\gamma\|^2))\\
&=&8\pi\epsilon^3/3-((\pi/6)(16\epsilon^3-12\|\gamma\|\epsilon^2+\|\gamma\|^3))\\
&=&2\pi\|\gamma\|\epsilon^2-\pi\|\gamma\|^3/6.
\end{eqnarray*}
Thus $$\ell(B_\epsilon(y+\gamma)\triangle B_\epsilon(y))=\left\{
                                                   \begin{array}{ll}
                                                     (2\pi\|\gamma\|\epsilon^2-\pi\|\gamma\|^3/6), & \hbox{$\|\gamma\|\le 2\epsilon$;} \\
                                                     8\pi\epsilon^3/3, & \hbox{$\|\gamma\|>2\epsilon$.}
                                                   \end{array}
                                                 \right..$$
Thus \begin{eqnarray*}
\left(\ell(B_\epsilon(y+\gamma)\triangle B_\epsilon(y))\right)^{1/2}/\ell(B_\epsilon(y))&=&\left\{
                          \begin{array}{ll}
                            \sqrt{(2\pi\|\gamma\|\epsilon^2-\pi\|\gamma\|^3/6)}\cdot 3/(4\pi\epsilon^3), & \hbox{$\|\gamma\|\le 2\epsilon$;} \\
                            \sqrt{3/2\pi}\frac{\sqrt{\epsilon}}{\epsilon^2}, & \hbox{$\|\gamma\|>2\epsilon$.}
                          \end{array}
                        \right.,\\
&=&\left\{
                          \begin{array}{ll}
                            \sqrt{1/\pi}\sqrt{9\|\gamma\|/8-3\|\gamma\|^3/32\epsilon^2}\frac{1}{\epsilon^2}, & \hbox{$\|\gamma\|\le 2\epsilon$;} \\
                            \sqrt{3/2\pi}\frac{\sqrt{\epsilon}}{\epsilon^2}, & \hbox{$\|\gamma\|>2\epsilon$.}
                          \end{array}
                        \right.,
                        \end{eqnarray*}
and one has $$|\Dx f(x)-\Dx f(y)|\le \|f\|_{L^2(\ell)}\cdot(3/2)\sqrt{1/2\pi}\frac{1}{\epsilon^2}\|x-y\|^{1/2}$$ for all $x,y\in X$.
\end{proof}

Setting notation for the following lemmas, $T:M\to M$ is a diffeomorphism with $0<A\le |\det DT|\le B<\infty$ and $0<L\le h_\mu\le U<\infty$.
Moreover $X=Y=Y_\epsilon=M$.

\begin{lemma}
\label{lemB0}
$\|f\cdot h_\mu\|_{L^2(\ell)}\le U^{1/2}\|f\|_{L^2(\mu)}$.
\end{lemma}
\begin{proof}
\begin{equation*}
\|f\cdot h_\mu\|_{L^2(\ell)}^2=\int f^2\cdot h_\mu^2\ d\ell=\int f^2\cdot h_\mu\ d\mu\le U\int f^2\ d\mu= U\|f\|_{L^2(\mu)}^2
\end{equation*}
\end{proof}

\begin{lemma}
\label{averagelemma}
If $f\in L^1(\ell)$ satisfies $c\le f\le d$ then $c\le \Dx f\le d$
\end{lemma}
\begin{proof}
This follows since $\Dx$ is an averaging operator.
\end{proof}

\begin{lemma}
\label{lemD1}$
\|\Dx f\|_{L^2(\ell)}\le \|f\|_{L^2(\ell)}.$
\end{lemma}
\begin{proof}
This follows directly from Lemma \ref{Lbound}, putting $k(x,y)=\textbf{1}_{B_\epsilon(y)}(x)/\ell(B_\epsilon(0))$ one has $\int_X k(x,y)\ dx=1$ for all $y\in X_\epsilon$ and $\int_{X_\epsilon} k(x,y)\ dy\le 1$ for all $x\in X$.
\end{proof}

\begin{lemma}
\label{lemB1}
$\|\mathcal{P}f\|_{L^2(\ell)}\le 1/A^{1/2}\|f\|_{L^2(\ell)}.$
\end{lemma}
\begin{proof}
\begin{eqnarray*}
\int (\mathcal{P}f)^2\ d\ell&=&\int \left(\frac{f\circ T^{-1}}{|\det DT\circ T^{-1}|}\right)^2\ d\ell\\
&=&\int \left(\frac{f}{|\det DT|}\right)^2\cdot \frac{1}{|\det DT^{-1}\circ T|}\ d\ell\quad\mbox{by change of variables under $T$}\\
&=&\int \frac{f^2}{|\det DT|}\ d\ell\\
&\le& 1/A \int f^2\ d\ell
\end{eqnarray*}
\end{proof}

\begin{proof}[Proof of Proposition \ref{Lholderlemma}]
\quad
\begin{enumerate}
\item
Let $H_{1/2}(f)$ denote the $1/2$-H\"older exponent for $f$.
Since $\mathcal{L}_\epsilon f=\Dy\mathcal{P}\Dx(f\cdot h_\mu)/\Dy\mathcal{P}\Dx h_\mu$ we have
\begin{eqnarray*}
\lefteqn{H_{1/2}(\mathcal{L}_\epsilon f)\le H_{1/2}(\Dy\mathcal{P}\Dx(f\cdot h_\mu))\cdot \left|\frac{1}{\Dy\mathcal{P}\Dx h_\mu}\right|_\infty}\\
&&+H_{1/2}\left(\frac{1}{\Dy\mathcal{P}\Dx h_\mu}\right)\cdot |\Dy\mathcal{P}\Dx(f\cdot h_\mu))|_\infty.
\end{eqnarray*}
By
Lemmas \ref{lemB0}, \ref{lemD1}, and \ref{lemB1} we have
$$\|\mathcal{P}\Dx(f\cdot h_\mu)\|_{L^2(\ell)}\le \|\mathcal{P}\|_{L^2(\ell)}\|\Dx(f\cdot h_\mu)\|_{L^2(\ell)}\le (U/A)^{1/2}\|f\|_{L^2(\mu)}$$ and so applying Lemma \ref{holderlemma} we have $$H_{1/2}(\Dy\mathcal{P}\Dx(f\cdot h_\mu))\le (U/A)^{1/2}\|f\|_{L^2(\mu)}C(d)/\epsilon^{(d+1)/2}$$ where $C(d)$ is from Lemma \ref{holderlemma}.
Moreover, since $h_{\mu_\epsilon}=\Dx h_\mu$ is bounded below and above by $L$ and $U$, respectively (by Lemma \ref{averagelemma}) we have $h_{\nu'_\epsilon}=\mathcal{P}h_{\mu_\epsilon}$ is bounded below and above by $L/B$ and $U/A$ respectively.
Finally, applying Lemma \ref{averagelemma} again, we have $A/U\le |1/\Dy \mathcal{P}\Dx h_\mu|\le B/L$.

For the second term, note
\begin{eqnarray*}
H_{1/2}\left(\frac{1}{\Dy h_{\nu'_\epsilon}}\right)&\le&\frac{H_{1/2}(\Dy h_{\nu'_\epsilon})}{(\min \Dy h_{\nu'_\epsilon})^2}\\
&\le&\frac{\|h_{\nu'_\epsilon}\|_{L^2(\ell)}}{(\min h_{\nu'_\epsilon})^2}\cdot\frac{C(d)}{\epsilon^{(d+1)/2}}\qquad\mbox{by Lemma \ref{holderlemma} and Lemma \ref{averagelemma}}\\
&\le&\frac{(\max h_{\nu'_\epsilon})^{1/2}}{(\min h_{\nu'_\epsilon})^2}\cdot\frac{C(d)}{\epsilon^{(d+1)/2}}\qquad\mbox{as $\|h_{\nu'_\epsilon}\|_{L^1(\ell)}=1$ we have $\|h_{\nu'_\epsilon}\|_{L^2(\ell)}\le (\max h_{\nu'_\epsilon})^{1/2}$} \\
&\le& \frac{(U/A)^{1/2}}{(L/B)^2}\cdot C(d)/\epsilon^{(d+1)/2}
\end{eqnarray*}
Finally, to bound $|\Dy\mathcal{P}\Dx(f\cdot h_\mu))|_\infty$ we note that since $\int f\ d\mu=0$ we have $\int \Dy\mathcal{P}\Dx(f\cdot h_\mu)\ d\ell=0$ since $\Dy, \Dx,$ and $\mathcal{P}$ preserve $\ell$-integrals.
Thus, $|\Dy\mathcal{P}\Dx(f\cdot h_\mu))|_\infty\le H_{1/2}(\Dy\mathcal{P}\Dx(f\cdot h_\mu))|\diam(M)|^{1/2}$.

Putting this all together we have
\begin{eqnarray*}
\lefteqn{H_{1/2}(\mathcal{L}_\epsilon f)\le(U/A)^{1/2}\ell(X_\epsilon)\|\alpha\|_\infty\cdot(B/L)\|f\|_{L^2(\mu)}C(d)/\epsilon^{(d+1)/2}}\\
&&+\frac{(U/A)^{1/2}}{(L/B)^2}\cdot C(d)/\epsilon^{(d+1)/2}\cdot (U/A)^{1/2}\|f\|_{L^2(\mu)}C(d)/\epsilon^{(d+1)/2} |\diam(M)|^{1/2}.
\end{eqnarray*}
\item
Let $L(f)$ denote the Lipschitz exponent for $f$.
Since $\mathcal{L}_\epsilon f=\Dy\mathcal{P}\Dx(f\cdot h_\mu)/\Dy\mathcal{P}\Dx h_\mu$ we have
$$L(\mathcal{L}_\epsilon f)\le L(\Dy\mathcal{P}\Dx(f\cdot h_\mu))\cdot \left|\frac{1}{\Dy\mathcal{P}\Dx h_\mu}\right|_\infty+L\left(\frac{1}{\Dy\mathcal{P}\Dx h_\mu}\right)\cdot |\Dy\mathcal{P}\Dx(f\cdot h_\mu))|_\infty.$$

 We have that $|\mathcal{P}\Dx(f\cdot h_\mu)|\le (U/A)\|f\|_\infty$ and applying Lemma \ref{lipschitzlemma} we have $L(\Dy\mathcal{P}\Dx(f\cdot h_\mu))\le (U/A)\|f\|_\infty C_L(d)/\epsilon$ where $C_L(d)$ is from Lemma \ref{lipschitzlemma}.
Similarly, we have $A/U\le |1/\Dy \mathcal{P}\Dx h_\mu|\le B/L$.

For the second term, note
\begin{eqnarray*}
L\left(\frac{1}{\Dy h_{\nu'_\epsilon}}\right)&\le&\frac{L(\Dy h_{\nu'_\epsilon})}{(\min \Dy h_{\nu'_\epsilon})^2}\\
&\le&\frac{\max h_{\nu'_\epsilon}}{(\min h_{\nu'_\epsilon})^2}\cdot\frac{C_L(d)}{\epsilon}\qquad\mbox{by Lemma \ref{lipschitzlemma} and Lemma \ref{averagelemma}}\\
&\le& \frac{(U/A)}{(L/B)^2}\cdot C_L(d)/\epsilon
\end{eqnarray*}
Finally, to bound $|\Dy\mathcal{P}\Dx(f\cdot h_\mu))|_\infty$ we note that since $\int f\ d\mu=0$ we have $\int \Dy\mathcal{P}\Dx(f\cdot h_\mu)\ d\ell=0$ since $\Dx, \Dy,$ and $\mathcal{P}$ preserve $\ell$-integrals.
Thus, $|\Dy\mathcal{P}\Dx(f\cdot h_\mu))|_\infty\le L(\Dy\mathcal{P}\Dx(f\cdot h_\mu))|\diam(M)|$.

Putting this all together we have
\begin{equation*}
L(\mathcal{L}_\epsilon f)\le (U/A)(B/L)\|f\|_\infty C_L(d)/\epsilon+\frac{(U/A)^2}{(L/B)^2}\|f\|_\infty  |\diam(M)| C_L(d)^2/\epsilon^2
\end{equation*}
\end{enumerate}
\end{proof}

\begin{lemma}
\label{koopmanbound}
Consider $\mathcal{K}:L^2(Y'_\epsilon,\ell)\to L^2(X_\epsilon,\ell)$. One has $\|\mathcal{K}g\|_{L^2(X_\epsilon,\ell)}\le \|g\|_{L^2(Y'_\epsilon,\ell)}/\sqrt{A}$.
\end{lemma}
\begin{proof}
\begin{eqnarray*}
\|\mathcal{K}g\|_{L^2(X_\epsilon,\ell)}&=&\left(\int_{X_\epsilon} (g(Tx))^2\ dx\right)^{1/2} \\
&=&\left(\int_{Y'_\epsilon} (g(x))^2/|\det DT(T^{-1}x)|\ dx\right)^{1/2}\\
&\le&(1/\sqrt{A})\left(\int_{Y'_\epsilon} (g(x))^2\ dx\right)^{1/2}=\|g\|_{L^2(Y'_\epsilon,\ell)}/\sqrt{A}.
\end{eqnarray*}
\end{proof}

\begin{proof}[Proof of Proposition \ref{Lstarholderlemma}]
Note that $\mathcal{L}^*_\epsilon=\Dx^*\mathcal{K}\Dy^*$.
We first claim the $L^2$-norm of $\mathcal{K}\Dy^*$ is $1/\sqrt{A}$.
By Lemma \ref{Lbound}, putting $X=Y'_\epsilon$, $Y=Y_\epsilon$, and $k(x,y)=\mathbf{1}_{B_\epsilon(y)}(x)/\ell(B_\epsilon(y))$, we have that $\|\Dy g\|_{L^2(Y_\epsilon,\ell)}\le \|g\|_{L^2(Y'_\epsilon,\ell)}$.
This follows since $\int_{Y'_\epsilon} \mathbf{1}_{B_\epsilon(y)}(x)/\ell(B_\epsilon(y))\ dx\le 1$ for $y\in Y_\epsilon$, and $\int_{Y_\epsilon} \mathbf{1}_{B_\epsilon(x)}(y)/\ell(B_\epsilon(x))\ dy=1$ for $x\in Y'_\epsilon$.
By Lemma \ref{koopmanbound} the claim follows.

We now consider $|\Dx^* f(x)-\Dx^* f(y)|$ for $f\in L^2(X_\epsilon,\ell)$.
Because of the symmetry of the kernel $k(x,y)=\mathbf{1}_{B_\epsilon(y)}(x)/\ell(B_\epsilon(y))$, the only difference between $\Dx$ and $\Dx^*$ is the domain of integration (respectively $X$ and $X_\epsilon$).
The bound of Lemma \ref{holderlemma} thus also applies to $\Dx^*$.
Setting $f=\mathcal{K}\Dy^*g$ we have
$$|\Dx^*\mathcal{K}\Dy^*g(x)-\Dx^*\mathcal{K}\Dy^*g(y)|\le (C(d)/\sqrt{A})(1/\epsilon^{(1+d)/2})\|g\|_{L^2(Y_\epsilon,\ell)}\cdot\|x-y\|^{1/2}.$$
\end{proof}

%
%


\begin{lemma}
\label{localgloballipschitz}
Let $X\subset \mathbb{R}^n$ compact and $F:X\to \mathbb{R}$.
If $$\limsup_{\|\gamma\|\to 0}\frac{|F(x+\gamma)-F(x)|}{\|\gamma\|}\le C\mbox{ for all $x\in X$},$$
then $F$ is globally Lipschitz with Lipschitz constant $C$.
\end{lemma}
\begin{proof}
Note $\limsup_{\|\gamma\|\to 0}|F(x+\gamma)-F(x)|/\|\gamma\|=\lim_{\tilde{\gamma}\to 0}\sup\{|F(x+\gamma)-F(x)|/\|\gamma\|: \|\gamma\|<\tilde{\gamma},\gamma\neq 0\}.$
Thus given $\varepsilon>0$ there is a $\Delta=\Delta(\varepsilon,x)>0$ such that $|F(x+\gamma)-F(x)|/\|\gamma\|<C+\varepsilon$ for all $\gamma\neq 0$ with $\|\gamma\|<\Delta$.
Form an open cover of $X$ as $\{B_{\Delta(\varepsilon,x)}(x):x\in X\}$.
By compactness of $X$ we can find a finite subcover.

Consider arbitrary $x,y\in X$ and write $y=x+\gamma$ (note $\gamma$ is arbitrary from now on and need not satisfy $\|\gamma\|<\Delta$).
Draw a line segment from $x$ to $x+\gamma$;  this line segment $\{x+\theta\gamma: 0\le \theta\le 1\}$ passes through a subcollection of open balls from our finite subcover.
Denote by $X'$ the finite set of centres of the open balls in our subcollection.


We now trace out the line segment again, identifying a finite sequence of points $y_\ell$ and ball centres $x_\ell$ as we go.
Begin at $x$, set $y_0=x$ and choose an $x_0\in X'$ so that $x\in B_{\Delta(\varepsilon,x_0)}$.
Now increase $\theta$ until $\|x_0-(x+\theta\gamma)\|=0.9\Delta(\varepsilon,x_0)$.
If $x+\theta\gamma$ lies in some $B_{\Delta(\varepsilon,x_1)}, x_1\in X', x_1\neq x_0$, then set $y_1=x+\theta\gamma$;  otherwise, increase $\theta$ until $x+\theta\gamma$ lies in $B_{\Delta(\varepsilon,x_0)}\cap B_{\Delta(\varepsilon,x_1)}$, $x_1\in X'$, $x_1\neq x_0$ and set $y_1=x+\theta\gamma$.
Repeat the procedure:  in general, we have a $y_\ell\in B_{\Delta(\varepsilon,x_{\ell})}$ and we increase $\gamma$ until $\|x_{\ell}-(y_\ell+\theta\gamma)\|=0.9\Delta(\varepsilon,x_{\ell})$.
If $y_\ell+\theta\gamma$ lies in some $B_{\Delta(\varepsilon,x_{{\ell+1}})}, x_{\ell+1}\in X', x_{\ell+1}\neq x_\ell,\ldots,x_0$, then set $y_{\ell+1}=y_\ell+\theta\gamma$;  otherwise, increase $\theta$ until $y_\ell+\theta\gamma$ lies in $B_{\Delta(\varepsilon,x_{\ell})}\cap B_{\Delta(\varepsilon,x_{\ell+1})}$, $x_{\ell+1}\in X'$, $x_{\ell+1}\neq x_\ell,\ldots,x_0$, and again set $y_{\ell+1}=y_\ell+\theta\gamma$.
Finally we make a step from $y_{L-1}$ to $y_L:=x+\gamma$ where $y_{L-1},x+\gamma\in B_{\Delta(\varepsilon,x_{L-1})}$.

By construction, $y_0\in B_{\Delta(\varepsilon,x_{0})}$, $y_\ell\in B_{\Delta(\varepsilon,x_{\ell-1})}\cap B_{\Delta(\varepsilon,x_{\ell})}$ for $1\le\ell\le L-1$, and $y_L\in B_{\Delta(\varepsilon,x_{L-1})}$.
Thus $y_{\ell},y_{\ell+1}\in B_{\Delta(\varepsilon,x_{\ell})}$ for $\ell=0,\ldots,L-1$, we may estimate
\begin{eqnarray*}
|F(x)-F(x+\gamma)|&=&\left|\sum_{\ell=0}^{L-1} F(y_\ell)-F(y_{\ell+1})\right|\\
&\le&\sum_{\ell=0}^{L-1} |F(y_\ell)-F(y_{\ell+1})|\\
&\le&(C+\varepsilon)\sum_{\ell=0}^{L-1} \|y_\ell-y_{\ell+1}\|\\
&=&(C+\varepsilon)\|y_0-y_L\|\\
&=&(C+\varepsilon)\|\gamma\|,
\end{eqnarray*}
where the penultimate equality follows since the $y_{\ell}$ are collinear.
As $x,y\in X$ were arbitrary the result follows.
\end{proof}

\section{Acknowledgements}
GF is grateful to Kathrin Padberg-Gehle for feedback on an earlier draft, and to George Haller for posing the question of frame-invariance.


\begin{thebibliography}{10}

\bibitem{aref_02}
H.~Aref.
\newblock The development of chaotic advection.
\newblock {\em Physics of Fluids}, 14(4):1315--1325, 2002.

\bibitem{birman_solomjak}
M.S. Birman and M.Z. Solomjak.
\newblock {\em Spectral theory of self-adjoint operators in Hilbert space}.
\newblock D. Reidel Publishing Co., Inc., 1986.

\bibitem{bollt_billings_schwartz_02}
E.M. Bollt, L.~Billings, and I.B. Schwartz.
\newblock A manifold independent approach to understanding transport in
  stochastic dynamical systems.
\newblock {\em Physica D}, 173:153--177, 2002.

\bibitem{budivsic_2012}
M.~Budi{\v{s}}i{\'c} and I.~Mezi{\'c}.
\newblock Geometry of the ergodic quotient reveals coherent structures in
  flows.
\newblock {\em Physica D: Nonlinear Phenomena}, 2012.

\bibitem{cerbellietal}
S.~Cerbelli, V.~Vitacolonna, A.~Adrover, and M.~Giona.
\newblock Eigenvalue--eigenfunction analysis of infinitely fast reactions and
  micromixing regimes in regular and chaotic bounded flows.
\newblock {\em Chemical Engineering Science}, 59(11):2125--2144, 2004.

\bibitem{christov_etal_11}
I.C. Christov, J.M. Ottino, and R.M. Lueptow.
\newblock From streamline jumping to strange eigenmodes: Bridging the
  {L}agrangian and {E}ulerian pictures of the kinematics of mixing in granular
  flows.
\newblock {\em Physics of Fluids}, 23:103302, 2011.

\bibitem{conway}
J.B. Conway.
\newblock {\em A course in functional analysis}, volume~96 of {\em Graduate
  texts in mathematics}.
\newblock Springer, 2nd edition, 1990.

\bibitem{DFHPS09}
M.~Dellnitz, G.~Froyland, C.~Horenkamp, K.~Padberg-Gehle, and A.~Sen Gupta.
\newblock Seasonal variability of the subpolar gyres in the southern ocean: a
  numerical investigation based on transfer operators.
\newblock {\em Nonlinear Processes in Geophysics}, 16:655--664, 2009.

\bibitem{DFJ01}
M.~Dellnitz, G.~Froyland, and O.~Junge.
\newblock {The algorithms behind {\sc GAIO} -- {S}et oriented numerical methods
  for dynamical systems}.
\newblock In B.~Fiedler, editor, {\em Ergodic Theory, Analysis, and Efficient
  Simulation of Dynamical Systems}, pages 145--174. Springer, 2001.

\bibitem{dellnitz_junge_99}
M.~Dellnitz and O.~Junge.
\newblock On the approximation of complicated dynamical behaviour.
\newblock {\em SIAM Journal for Numerical Analysis}, 36(2):491--515, 1999.

\bibitem{deuflhard_weber_05}
P.~Deuflhard and M.~Weber.
\newblock Robust perron cluster analysis in conformation dynamics.
\newblock {\em Linear algebra and its applications}, 398:161--184, 2005.

\bibitem{england_rahmstorf_1999}
M.H. England and S.~Rahmstorf.
\newblock Sensitivity of ventilation rates and radiocarbon uptake to
  subgrid-scale mixing in ocean models.
\newblock {\em J. Phys. Oceanogr.}, 29:2802–--2828, 1999.

\bibitem{froyland_01}
G.~Froyland.
\newblock Extracting dynamical behaviour via {M}arkov models.
\newblock In Alistair~I.\ Mees, editor, {\em Nonlinear Dynamics and Statistics:
  Proceedings, Newton Institute, Cambridge, 1998}, pages 283--324.
  Birkh{\"a}user, 2001.

\bibitem{froyland_05}
G.~Froyland.
\newblock Statistically optimal almost-invariant sets.
\newblock {\em Physica D}, 200:205--219, 2005.

\bibitem{FD03}
G.~Froyland and M.~Dellnitz.
\newblock Detecting and locating near-optimal almost-invariant sets and cycles.
\newblock {\em SIAM J. Sci. Comput.}, 24(6):1839--1863, 2003.

\bibitem{FHRSS12}
G.~Froyland, C.~Horenkamp, V.~Rossi, N.~Santitissadeekorn, and A.~Sen Gupta.
\newblock Three-dimensional characterization and tracking of an {A}gulhas
  {R}ing.
\newblock {\em Ocean Modelling}, 52--53:69--75, 2012.

\bibitem{FLQ10}
G.~Froyland, S.~Lloyd, and A.~Quas.
\newblock {Coherent structures and isolated spectrum for Perron-Frobenius
  cocycles}.
\newblock {\em Ergodic Theory and Dynamical Systems}, 30:729--756, 2010.

\bibitem{FLQ2}
G.~Froyland, S.~Lloyd, and A.~Quas.
\newblock A semi-invertible oseledets theorem with applications to transfer
  operator cocycles.
\newblock arXiv preprint arXiv:1001.5313, 2010.

\bibitem{FLS10}
G.~Froyland, S.~Lloyd, and N.~Santitissadeekorn.
\newblock Coherent sets for nonautonomous dynamical systems.
\newblock {\em Physica D}, 239:1527--1541, 2010.

\bibitem{froyland_padberg_england_treguier_07}
G.~Froyland, K.~Padberg, M.H.\ England, and A.M. Treguier.
\newblock Detection of coherent oceanic structures via transfer operators.
\newblock {\em Physical Review Letters}, 98(22):224503, 2007.

\bibitem{froyland_padberg_12}
G.~Froyland and K.~Padberg-Gehle.
\newblock Finite-time entropy: A probabilistic approach for measuring nonlinear
  stretching.
\newblock {\em Physica D: Nonlinear Phenomena}, 241:1612--1628, 2012.

\bibitem{FSM10}
G.~Froyland, N.~Santitissadeekorn, and A.~Monahan.
\newblock Transport in time-dependent dynamical systems: Finite-time coherent
  sets.
\newblock {\em Chaos}, 20:043116, 2010.

\bibitem{gaveau_schulman_06}
B.~Gaveau and L.S. Schulman.
\newblock Multiple phases in stochastic dynamics: {G}eometry and probabilities.
\newblock {\em Physical Review E}, 73:036124, 2006.

\bibitem{GTQ}
C.~Gonz{\'{a}}lez-Tokman and A.~Quas.
\newblock A semi-invertible operator {O}seledets theorem.
\newblock {\em Ergodic Theory and Dynamical Systems}, 2012.
\newblock To appear.

\bibitem{anderson2012}
O.~Gorodetskyi, M.~Giona, and P.D. Anderson.
\newblock Spectral analysis of mixing in chaotic flows via the mapping matrix
  formalism: Inclusion of molecular diffusion and quantitative eigenvalue
  estimate in the purely convective limit.
\newblock {\em Physics of Fluids}, 24(7):073603, 2012.

\bibitem{Haller_00}
G.\ Haller.
\newblock Finding finite-time invariant manifolds in two-dimensional velocity
  fields.
\newblock {\em Chaos}, 10:99--–108, 2000.

\bibitem{Haller_01}
G.~Haller.
\newblock Distinguished material surfaces and coherent structures in
  three-dimensional fluid flows.
\newblock {\em Physica D}, 149:248--–277, 2001.

\bibitem{hallerobjective}
G.~Haller.
\newblock An objective definition of a vortex.
\newblock {\em Journal of Fluid Mechanics}, 525(1):1--26, 2005.

\bibitem{Haller_11}
G.\ Haller.
\newblock A variational theory of hyperbolic {L}agrangian coherent structures.
\newblock {\em Physica D}, 240:574--–598, 2011.

\bibitem{Haller_12}
G.~Haller and F.J. Beron-Vera.
\newblock Geodesic theory of transport barriers in two-dimensional flows.
\newblock {\em Physica D: Nonlinear Phenomena}, 241:1680--–1702, 2012.

\bibitem{huisinga_schmidt}
W.~Huisinga and B.~Schmidt.
\newblock {\em Metastability and dominant eigenvalues of transfer operators},
  volume~49 of {\em Lecture Notes in Computational Science and Engineering},
  pages 167--182.
\newblock Springer, 2006.

\bibitem{jungethesis}
O.~Junge.
\newblock {\em Mengenorientierte Methoden zur numerischen Analyse dynamischer
  Systeme}.
\newblock PhD thesis, Universit{\"a}t Paderborn, 2000.

\bibitem{jungemarsdenmezic}
O.~Junge, J.E. Marsden, and I.~Mezic.
\newblock Uncertainty in the dynamics of conservative maps.
\newblock In {\em 43rd IEEE Conference on Decision and Control, 2004},
  volume~2, pages 2225--2230, 2004.

\bibitem{knutti_et_al_2000}
R.~Knutti, T.F. Stocker, and D.G. Wright.
\newblock The effects of subgrid-scale parameterizations in a zonally averaged
  ocean model.
\newblock {\em Journal of physical oceanography}, 30(11):2738--2752, 2000.

\bibitem{lasota_mackey2}
A.~Lasota and M.C.\ Mackey.
\newblock {\em Chaos, Fractals, and Noise: Stochastic Aspects of Dynamics},
  volume~97 of {\em Applied Mathematical Sciences}.
\newblock Springer-Verlag, New York, 2nd edition, 1994.

\bibitem{Meiss1992}
J.~D. Meiss.
\newblock Symplectic maps, variational principles, and transport.
\newblock {\em Rev. Mod. Phys.}, 64(3):795--848, 1992.

\bibitem{mezic_wiggins_99}
I.~Mezi\'c and S.~Wiggins.
\newblock A method for visualization of invariant sets of dynamical systems
  based on the ergodic partition.
\newblock {\em Chaos}, 9(1):213--218, 1999.

\bibitem{padberg_thesis_05}
K.~Padberg.
\newblock {\em Numerical Analysis of Transport in Dynamical Systems}.
\newblock PhD thesis, Universit{\"a}t Paderborn, Paderborn, 2005.

\bibitem{romkedar_etal_1990}
V.~Rom-Kedar, A.~Leonard, and S.~Wiggins.
\newblock An analytical study of transport, mixing and chaos in an unsteady
  vortical flow.
\newblock {\em Journal of Fluid Mechanics}, 214:347--394, 1990.

\bibitem{romkedar_wiggins_90}
V.~Rom-Kedar and S.~Wiggins.
\newblock Transport in two-dimensional maps.
\newblock {\em Archive for Rational Mechanics and Analysis}, 109:239--298,
  1990.

\bibitem{rypina_etal_07}
I.I. Rypina, M.G. Brown, F.J. Beron-Vera, H.~Ko\c{c}ak, M.J. Olascoaga, and
  I.A. Udovydchenkov.
\newblock On the lagrangian dynamics of atmospheric zonal jets and the
  permeability of the stratospheric polar vortex.
\newblock {\em J. Atmos. Sci.}, 64:3595, 2007.

\bibitem{SFM10}
N.~Santitissadeekorn, G.~Froyland, and A.~Monahan.
\newblock Optimally coherent sets in geophysical flows: A new approach to
  delimiting the stratospheric polar vortex.
\newblock {\em Physical Review E}, 82:056311, 2010.

\bibitem{schuette_huisinga_deuflhard_01}
Ch. Sch\"{u}tte, W.~Huisinga, and P.~Deuflhard.
\newblock Transfer operator approach to conformational dynamics in biomolecular
  systems.
\newblock In Bernold Fiedler, editor, {\em Ergodic Theory, Analysis, and
  Efficient Simulation of Dynamical Systems}, pages 191--223. Springer, Berlin,
  2001.

\bibitem{shadden_lekien_marsden_05}
S.C. Shadden, F.~Lekien, and J.E. Marsden.
\newblock Definition and properties of {L}agrangian coherent structures from
  finite-time {L}yapunov exponents in two-dimensional aperiodic flows.
\newblock {\em Physica D}, 212:271--304, 2005.

\bibitem{singh_etal_09}
M.K. Singh, M.F.M. Speetjens, and P.D. Anderson.
\newblock Eigenmode analysis of scalar transport in distributive mixing.
\newblock {\em Physics of Fluids}, 21:093601--093601, 2009.

\bibitem{stremler11}
M.A. Stremler, S.D. Ross, P.~Grover, and P.~Kumar.
\newblock Topological chaos and periodic braiding of almost-cyclic sets.
\newblock {\em Physical Review Letters}, 106(11):114101, 2011.

\bibitem{truesdellnoll}
C.~Truesdell and W.~Noll.
\newblock {\em The non-linear field theories of mechanics}.
\newblock Springer, 3rd edition, 2004.

\bibitem{weisstein2}
E.W. Weisstein.
\newblock Circle-circle intersection.
\newblock From MathWorld--A Wolfram Web Resource.
\newblock http://mathworld.wolfram.com/Circle-CircleIntersection.html.

\bibitem{weisstein3}
E.W. Weisstein.
\newblock Sphere-sphere intersection.
\newblock From MathWorld--A Wolfram Web Resource.
\newblock http://mathworld.wolfram.com/Sphere-SphereIntersection.html.

\bibitem{wiggins_05}
S.~Wiggins.
\newblock The dynamical systems approach to {Lagrangian} transport in oceanic
  flows.
\newblock {\em Annu. Rev. Fluid Mech.}, 37:295--328, 2005.

\bibitem{zeeman}
E.C. Zeeman.
\newblock Stability of dynamical systems.
\newblock {\em Nonlinearity}, 1(1):115, 1999.

\end{thebibliography}
\end{document}